\documentclass[leqno]{siamart1116}
\usepackage{amssymb}
\usepackage{graphicx}
\usepackage{xspace}
\usepackage{multirow}
\usepackage{siunitx}
\usepackage{bm,bbm}
\usepackage{Macros/mydef}
\usepackage[numbers]{natbib}
\let\cite=\citet
\usepackage{enumerate}
\usepackage{dsfont}
\usepackage{enumitem}
\usepackage{subcaption} % for use of subfigures

\usepackage[algo2e,linesnumbered]{algorithm2e}
\usepackage{algorithmic}

\begin{document}

\newcommand\footnotemarkfromtitle[1]{%
\renewcommand{\thefootnote}{\fnsymbol{footnote}}%
\footnotemark[#1]%
\renewcommand{\thefootnote}{\arabic{footnote}}}

% Declare title and authors, without \thanks
\newcommand{\TheTitle}{Hyperbolic relaxation technique for solving the dispersive Serre--Green--Naghdi
equations with topography}
% authors are now in alphabetic order
\newcommand{\TheAuthors}{J.-L. Guermond, C. Kees, B. Popov, E. Tovar}

% Sets running headers as well as PDF title and authors
\headers{Hyperbolic relaxation technique}{\TheAuthors}

% Title. 
\title{{\TheTitle}\thanks{Draft version, \today \funding{This material
      is based upon work supported in part by the National Science
      Foundation grants DMS-1619892 and DMS-1620058, by the Air Force
      Office of Scientific Research, USAF, under grant/contract number
      FA9550-18-1-0397, and by the Army Research Office under
      grant/contract number W911NF-19-1-0431. Permission was granted
      by the Chief of Engineers to publish this information.}}}

% Authors: full names plus addresses.
\author{Jean-Luc Guermond\footnotemark[3]
\and Chris Kees\footnotemark[4]
\and Bojan Popov\footnotemark[3]
\and Eric Tovar\footnotemark[3]\footnotemark[2]
}

\maketitle

\renewcommand{\thefootnote}{\fnsymbol{footnote}}

\footnotetext[2]{U.S. Army Engineer Research and Development Center, Coastal and Hydraulics Laboratory (ERDC-CHL), Vicksburg, MS 39180, USA.}

\footnotetext[3]{Department of Mathematics, Texas A\&M University 3368
  TAMU, College Station, TX 77843, USA.}

\footnotetext[4]{Department of Civil \& Environmental Engineering, Louisiana State University
  3255 Patrick F. Taylor, Baton Rouge, LA 70803, USA.}

\renewcommand{\thefootnote}{\arabic{footnote}}

\begin{abstract}
  \bal{The objective of this paper is to propose a hyperbolic
    relaxation technique for the dispersive Serre--Green--Naghdi
    equations (also known as the fully non-linear Boussinesq
    equations) with full topography effects introduced
    in~\cite{green_naghdi_1976}
    and~\cite{seabra-santos_renouard_temperville_1987}.  This is done
    by revisiting a similar relaxation technique introduced in
    \cite{Gu_Po_To_Ke_FAKE_SGN_2019} with partial topography effects.}
  We also derive a family of analytical solutions for the
  one-dimensional dispersive Serre--Green--Naghdi equations that are
  used to \bal{verify the correctness} the proposed relaxed model. The
  method is then numerically illustrated \bal{and validated} by
  comparison with experimental results.
\end{abstract}

\begin{keywords}
  Shallow water, dispersive Serre equations, Serre--Green--Nagdhi, hyperbolic relaxation,
  well-balanced approximation, invariant domain, second-order accuracy, finite element method, positivity-preserving
\end{keywords}

\begin{AMS}
65M60, 65M12, 35L50, 35L65, 76M10
\end{AMS}

\section{Introduction} \label{Sec:introduction} The Saint\,-Venant
shallow water equations model the free surface of a body of water
evolving under the action of gravity under the assumptions that the
deformation of the free surface is small compared to the water
elevation and the bottom topography $z$ varies slowly. Letting
$\mu=\frac{\waterh_0^2}{L^2}$ be the shallowness parameter, where
$\waterh_0$ is the typical water height and $L$ is the typical
horizontal lengthscale, Saint-Venant's shallow water model is the
$\calO(\mu)$ approximation of the free surface Euler equations. This
is a hyperbolic system that has many practical applications, but one
of its main deficiencies is that it does not have dispersive
properties. In particular, it does not support \bal{smooth} solitary waves.  \bal{A
model that includes all the first-order correction in terms of the
shallowness parameter (\ie it is a $\calO(\mu^2)$ approximation of the
free surface incompressible Euler equations) has been introduced by
\cite[Eq.~(22), p.~860]{Serre_1953} with a flat topography. This model has been rediscovered verbatim in
\cite{Su_Gardner_1969}, and rediscovered again in
\cite{Green_Naghdi_1974} also with flat topography.
% (Why Serre's paper has not been given all
% the credit it deserves is unclear, since Eq.~(22) in Serre's paper is
% not a side note or an accident. As clearly stated by Serre right after
% Eq.~(22), his objective was to derive ``a generalization of the
% equations of Barr\'e-de-Saint-Venant that accounts for the slope and
% the curvature of the free surface''.)
The key property of this model
(as unequivocally recognized and illustrated in
\citep[\S2]{Serre_1953}) is that it supports smooth solitary waves. The model
has been further improved in~\cite[Eq.~(13)]{seabra-santos_renouard_temperville_1987}
and~\cite[Eq.~(4.27)--(4.31)]{green_naghdi_1976}
to include the effects of topography up to the order $\calO(\mu^2)$.}  We refer the
reader to \cite{Barthelemy_2004} and \cite{Lannes_2020} for
comprehensive reviews of the properties of the model.  All these works
based on informal asymptotic expansions have been rigorously
formalised in \cite[Thm~6.2]{Alvarez-Samaniego_Lannes_2008B}.
\bal{For brevity, we refer to the Serre--Green--Nagdhi equations as just the Serre
model.}
% Although the model is often referred to in the literature as the
% Serre--Green--Nagdhi equations or, for unclear reasons, just the
% Green--Nagdhi equations, we are going to call it in the present work
% the dispersive Serre model.

One major drawback of the dispersive Serre model from a numerical
perspective is that it involves third-order derivatives in space. The
presence of the third-order derivatives rules out any approximation
technique that is explicit in time, since this would require the time
step $\dt$ to behave like $\calO(h^3)V^{-1}L^{-2}$, where $h$ is the
mesh-size, $V$ is a characteristic wave speed scale, and $L$ is a
characteristic length scale. There are currently two popular classes
of techniques for addressing this difficulty. The first one is based
on Stang's operator splitting and combines explicit and implicit time
stepping, see for instance
\cite{Bonneton_etal_2011,samii_dawson_2020,Duran_Marche_2017}. Another
approach consists of reinterpreting the dispersive system as a
constrained first-order system and then relaxing the constraint,
\bal{see for instance
  \cite{Favrie_Gavrilyuk_2017},~\cite{TKACHENKO_thesis},~\cite{Gu_Po_To_Ke_FAKE_SGN_2019},
  and~\cite{Escalante_Dumbser_Castro_2019}.} Note that the technique
in~\cite{Escalante_Dumbser_Castro_2019} is based on a dispersive model
introduced in \cite{Bristeau_Mangenay_SaniteMarie_Seguin_2015} which
differs from the Serre model up to a multiplicative constant when the
topography is flat (their pressure constant is $\tfrac14$ as opposed
to the Serre constant $\tfrac13$). \bal{All the above relaxation
  techniques are, however, incomplete since they ignore the topography
  corrections introduced
  in~\citep{green_naghdi_1976,seabra-santos_renouard_temperville_1987}.
  These terms play an important role in modeling the effects of
  vertical acceleration due to the topography and are necessary for
  correctly reproducing laboratory experiments. \bal{A recent
    hyperbolic reformulation with topography effects was introduced
    in~\cite{dumbser_2020} which was based on the work of
    ~\cite{nieto_parisot_2018} (where the reformulation uses a
    constraint on the divergence of the velocity).  }
% where the topography corrections are
% accounted for and a hyperbolic reformulation different from ours is
% considered.
}

In the present paper we exclusively focus on the topography
  corrections. By revisiting the experiments reported in
  \cite{seabra-santos_renouard_temperville_1987}, we unambiguously
  demonstrate that the topography corrections are indeed important to
  reproduce experiment involving reflected waves (see
  Table~\ref{table:shelf_trans}). We also show that the Serre model
  with topography effects can be reformulated as a constrained
  first-order system. We propose a relaxation technique that makes the
  system hyperbolic and, thereby, allows for explicit time stepping
  under reasonable time step restrictions. The key difference with our
  previous work in \citep{Gu_Po_To_Ke_FAKE_SGN_2019} is now the
  presence of the topography effects in the pressure and extra
  topography source terms which make the analysis more involved: there is one
  more conservation equation and one more relaxation term has to be
  added; the source terms in the additional conservation equations
  have to be modified appropriately to be compatible with the global
  energy conservation equation.

The paper is organized as follows. In Section
  \ref{sec:DSW_model}, we recall the dispersive Serre model with
  topography and state the corresponding energy conservation equation.
  We also derive a family of analytical solutions to the
  one-dimensional system.  To the best of our knowledge, it is the
  first time that an exact solution to this nonlinear system of
  equations is proposed when the topography is nontrivial. This
  solution is not a manufactured solution since it does not involve
  any ad hoc source terms. The main interest of this exact solution is
  to help verify numerical codes for the approximation of the
  dispersive Serre model with topography (which, to the best of our
  knowledge, was done in the literature only with manufactured
  solutions involving ad hoc source terms).  In Section
  \ref{sec:reformulation}, we reformulate the dispersive Serre model
  as a first-order system of nonlinear conservation equations with two
  algebraic constraints. We then relax these constraints and propose a
  hyperbolic system of equations that allows for explicit time
  stepping and is compatible with dry states. This system is shown to
  admit an energy equation.  Finally, in Section
  \ref{Sec:numerical_illustrations}, we illustrate the proposed model.
  Using a continuous finite element technique, we compare the new
  method with that described in \citep{Gu_Po_To_Ke_FAKE_SGN_2019}
  \bal{(where the model is incomplete because the topography correction
  terms from \citep{green_naghdi_1976,seabra-santos_renouard_temperville_1987} are not
  accounted for)}. We also compare the new method with several
  laboratory experiments in one and two spatial dimensions and
  unequivocally conclude that the topography effects are
  important to reproduce the experiments.

%%%
\section{Dispersive Serre model with topography}
\label{sec:DSW_model}
In this section we recall the dispersive Serre model with topography
and \bal{derive} an exact solution to the one-dimensional steady state problem.

%%%
\subsection{The dispersive model}
Let $\bu=(\waterh,\bq)\tr$ be the dependent variable, where
$\waterh$ is the water height and $\bq$ is the flow rate, or
discharge. \bal{The dispersive Serre model with topography effects
is written as follows:}
% The dispersive Serre model without topography effects formulated in
% \cite{Serre_1953} and improved in
% \cite{seabra-santos_renouard_temperville_1987} to account for the topography effects
%
\begin{equation}
\partial_t \bu +\DIV\polf(\bu) + \bb(\bu,\GRAD z) = \bzero ,\quad
\text{\ae\ }\bx\in D,\ t\in \polR_+.
\label{shallow_equations}
\end{equation}
The flux $\polf(\bu)$ and the bathymetry source
$\bb(\bu,\GRAD z)$ are given by
\begin{equation}
\quad
\polf(\bu):=\left(\begin{matrix}
\bq\tr \\
\frac{1}{\waterh}  \bq {\otimes}\bq + p(\bu) \polI_d
\end{matrix}\right)\in \polR^{(1+d)\times d},\quad
\bb(\bu,\GRAD z) := \left(\begin{matrix} 0 \\ r(\bu)\GRAD z\end{matrix}\right),
\label{shallow_equations_def_fluxes}
\end{equation}
with the pressure $p(\bu)$ and the source $r(\bu)$ defined by
\begin{subequations} \label{pressure_topography_SGN}
\begin{align}
&p(\bu)  \eqq \tfrac12 g \waterh^2 +\waterh ^2\Big(  \tfrac13 \ddot{\waterh} + \tfrac12\dot\sfk\Big), &&
\dot{\waterh}\eqq \partial_t \waterh + \bv\ADV  \waterh,\qquad
\ddot{\waterh}\eqq \partial_t \dot{\waterh} + \bv\ADV \dot{\waterh}, \label{SGN_pressure}\\
&r(\bu) = g\waterh + \waterh\Big(\tfrac12 \ddot{\waterh} + \dot\sfk\Big), \label{SGN_topography}
&& \dot\sfk \eqq \partial_t(\bv\ADV z) + \bv\ADV(\bv\ADV z).
\end{align}
\end{subequations}
Here the vector field $\bv$ is the velocity and is defined by
$\bv:=\frac{1}{\waterh} \bq$.  Due to the presence of the term
$\tfrac12 \ddot{\waterh} + \dot\sfk$, the pressure mapping
$\bu\mapsto p(\bu)$ is not a function but a second-order differential
operator in space and time.  This implies that
\eqref{shallow_equations}--\eqref{pressure_topography_SGN} is not a
quasilinear first-order system. \bal{In this paper we exclusively
focus our attention on the terms $\tfrac12\waterh ^2\dot\sfk$ in
$p(\bu)$ and $\dot\sfk$ in $\br(\bu)$, which are the $\calO(\mu^2)$
contributions induced by the topography identified in \cite{green_naghdi_1976} and
\cite{seabra-santos_renouard_temperville_1987}.}

Notice that the mass conservation equation implies
$\dot\waterh = -\waterh \DIV \bv$, which in turn implies that
$\ddot{\waterh} = -\partial_t (\waterh \DIV \bv) - \bv\ADV(\waterh
\DIV \bv)$.
\bal{This means that the term $\waterh ^2 \ddot{\waterh}$ in the pressure
equation~\eqref{SGN_pressure} can also be rewritten as
$-\waterh^3(\partial_{t}\DIV\bv + \bv\ADV(\DIV\bv) -( \DIV\bv)^2)$,
which is exactly the expression given in \cite[Eq.~(22), p.~860]{Serre_1953}
and \cite[A14]{Su_Gardner_1969} twenty-one and five years earlier,
respectively, than \cite[Eq.~(4.16)--(4.19)]{Green_Naghdi_1974}.}
% One may then wonder why the system
% \eqref{shallow_equations}--\eqref{pressure_topography_SGN}
% is often referred to in the literature as just the Green--Naghdi
% equations.

One important result we are going to use is that the system
\eqref{shallow_equations}--\eqref{pressure_topography_SGN}
admits an energy equation if the solution is smooth.
% (otherwise it is an entropy inequality).
\begin{lemma}\label{lem:energy}
  Let $\bu$ be a smooth solution to
  \eqref{shallow_equations}--\eqref{pressure_topography_SGN},
  then the following holds true:
$\partial_t \calE(\bu) + \DIV(\bcalF(\bu)) =0$,
with
\begin{subequations}
\begin{align}
\calE(\bu) &\eqq \tfrac12 g \bal{(\waterh+ z)^2}  + \tfrac12 \waterh \bv^2
+ \tfrac16 \waterh  \Big(\big(\dot{\waterh} + \tfrac32 (\bv\ADV z)\big)^2 +\tfrac34 (\bv\ADV z)^2 \Big), \label{DSV_energy}\\
\bcalF(\bu) &\eqq \bv (\calE(\bu) \bal{-\tfrac12 g z^2} + p(\bu)).
\end{align}
\end{subequations}
\end{lemma}
\begin{proof} We direct the reader to the proof of Lemma~\ref{Lem:relax_energy} for a
  detailed process on the derivation of the energy equation. Note that the energy
  \eqref{DSV_energy} is the same as the one reported in \cite[Eq.~(1.19)--(1.22)]{castro_lannes_2014}.
\end{proof}
%%%
\subsection{One-dimensional steady-state solution}  \label{Sec:steady_state_solution}
We now
propose a steady-state solution to
\eqref{shallow_equations}--\eqref{pressure_topography_SGN}
with non-trivial topography.  We restrict ourselves to one space dimension
and assume that the solution to
\eqref{shallow_equations}--\eqref{pressure_topography_SGN}
is time-independent and smooth. The following assertion is the main
result from this section.
\begin{lemma}
\label{Lem:steady_state_solution}
Let $q\in\Real$, $a,r\in \Real_+$ and let the bathymetry profile be defined by
$z(x)\eqq -\tfrac12\frac{a}{(\cosh(rx))^{2}}$. Then
$\waterh(x) = \waterh_0(1+\frac{a}{(\cosh(rx))^{2}})$ with the
constant discharge $q$ is a steady state solution to
\eqref{shallow_equations}--\eqref{pressure_topography_SGN} if
%\Question{(JLG) I have added $\pm$.}
% I have changed q to be real number (BP)
\begin{equation}
 q := \pm\sqrt\frac{(1 + a)g \waterh_0^3}{2}, \qquad
r := \frac{1}{\waterh_0}\sqrt{\frac{3a}{(1 + a)}}. \label{eq:Lem:steady_state_solution}
\end{equation}
\end{lemma}
\begin{proof} Notice that the discharge $q$ is constant since the
  solution does not depend on time.  A steady-state solution can be
  found by solving the steady-state problem of the energy equation in
  Lemma~\ref{lem:energy}; this yields a Bernoulli-like relation
    for the dispersive Serre model
    \eqref{shallow_equations}--\eqref{pressure_topography_SGN}. More
    precisely, from Lemma~\ref{lem:energy}, we infer that
  $\partial_x (\calF(\bu))=0$, which implies
  $\calF(\bu(x))= C_{\text{Ber}}g q$ where $C_{\text{Ber}}$ is the
  Bernoulli constant. We look for a stationary wave with the following
  structure $\waterh(x) = \waterh_0(1+\frac{a}{(\cosh(rx))^{2}})$ and
  posit that the topography is of the form
  $z(x) = \lambda (\waterh(x)-\waterh_0)$. The problem now consists of
  finding relations between the parameters $a$, $r$, $\waterh_0$, $g$,
  and $\lambda$ so that the condition
  $g^{-1}q^{-1}\calF(\bu(x))= C_{\text{Ber}}$ is satisfied, \ie
\[
\waterh(1+\lambda) + \frac{q^2}{2g\waterh^2}
- \frac{q^2}{6g\waterh^2}(1-3\lambda^2)(\partial_x\waterh)^2
+ \frac{q^2}{3g\waterh}(1+\tfrac32\lambda) \partial_{xx}\waterh = C_{\text{Ber}} + \lambda \waterh_0,
\]
where we used $v = \frac{q}{\waterh}$.
By taking the limit of this identity for $|x|\to \infty$, we find that
$C_{\text{Ber}} = \waterh_0 + \frac{q^2}{2g\waterh_0^2}.$ After
inserting the ansatz $\waterh(x)=\waterh_0(1+\frac{a}{(\cosh(rx))^{2}})$
into the above identity, we find that the following must hold true for
all $x\in \Real$:
\begin{multline*}
\big((1 + \lambda)g\waterh_0^3 + (\lambda  + \tfrac23)2r^2q^2\waterh_0^2 - q^2\big)\cosh(rx)^4 \\
+ \big((1 + \lambda)2ag\waterh_0^3 + ((\lambda^2 + \lambda + \tfrac13)a - \tfrac32 \lambda - 1)2r^2q^2\waterh_0^2
- \tfrac12 aq^2\big)\cosh(rx)^2 \\
+ \big((1 + \lambda)ag\waterh_0 - 2(\lambda^2 + \tfrac32 \lambda + \tfrac23)r^2q^2\big)a\waterh_0^2 =0.
\end{multline*}
This is equivalent to asserting that following nonlinear system of equations has a solution:
\begin{align*}
 (1 + \lambda)ag\waterh_0 - 2(\lambda^2 + \tfrac32 \lambda + \tfrac23)r^2q^2 &= 0,\\
(1 + \lambda)2ag\waterh_0^3 + ((\lambda^2 + \lambda + \tfrac13)a - \tfrac32 \lambda - 1)2r^2q^2\waterh_0^2
- \tfrac12 aq^2 &= 0,\\
(1 + \lambda)g\waterh_0^3 + (\lambda  + \tfrac23)2r^2q^2\waterh_0^2 - q^2 &=0.
\end{align*}
%Up to evident sign permutations on $q$,
The only nontrivial solution we have found
is $\lambda=-\tfrac12$ with $q$ and $r$ satisfying
\eqref{eq:Lem:steady_state_solution}.
\end{proof}
It seems that the solution formulated in
Lemma~\ref{Lem:steady_state_solution} is the first exact solution to
the nonlinear system
\eqref{shallow_equations}--\eqref{pressure_topography_SGN}
proposed in the literature when the topography is nontrivial. Our
motivation to construct this solution was mainly to verify convergence of numerical codes
solving
\eqref{shallow_equations}--\eqref{pressure_topography_SGN}.

\section{Reformulation of the dispersive Serre model}
\label{sec:reformulation}
In this section, we reformulate the dispersive Serre model
as a first-order system under two algebraic constraints. Our goal is to
introduce a relaxation technique in the spirit of
%\cite{Favrie_Gavrilyuk_2017} and its revised version described in
\cite{Gu_Po_To_Ke_FAKE_SGN_2019}.
%%%
\subsection{Reformulation} The analysis done in
%\citep{Favrie_Gavrilyuk_2017} and
\citep{Gu_Po_To_Ke_FAKE_SGN_2019} is incomplete
since it does not account for the dispersion terms induced by
the topography. We now revisit
\citep{Gu_Po_To_Ke_FAKE_SGN_2019} to fill in the gap and account for the missing terms. The main result of
this section is the following.

\begin{lemma} \label{lem:equivalence} Let
  $\bu:\Dom\CROSS (0,T)\to \Real_+\CROSS \Real^d$ be a smooth
  function.  Then $\bu$ solves the dispersive Serre
  model~\eqref{shallow_equations}--\eqref{pressure_topography_SGN}
  iff $(\bu,q_1,q_2,q_3)$ solves
\begin{subequations}\label{eq:lem:equivalence}
\begin{align}
&\partial_t \waterh + \DIV\bq = 0,\label{mass:lem:equivalence}\\\
&\partial_t \bq + \DIV(\bv\otimes \bq) + \GRAD (\tfrac12g\waterh^2 - \tfrac13 \waterh s)
=-(g \waterh -\tfrac12 s + \tfrac14 \ts )\GRAD z,\label{q:lem:equivalence}\\
&\partial_t q_1 + \DIV (\bv q_1) = q_2 -\tfrac32 q_3,\label{q1:lem:equivalence}\\
&\partial_t q_2 + \DIV (\bv q_2) = -s, \label{q2:lem:equivalence}\\
&\partial_t q_3 + \DIV (\bv q_3) = \ts, \label{q3:lem:equivalence}\\
&q_1 = \waterh^2,\quad q_3 = \bq\ADV z.  \label{q1q3:lem:equivalence}
\end{align}
\end{subequations}
\end{lemma}
\begin{proof}
  Assume that $\bu$ solves
  \eqref{shallow_equations}--\eqref{pressure_topography_SGN}. We
  are going to show that $\bu\eqq(\waterh,\bq)$
  solves \eqref{mass:lem:equivalence}-\eqref{q1q3:lem:equivalence}. \\
  \textup{(i)} Let $q_1 \eqq \waterh^2$ and $q_3 \eqq \bq\ADV z$ be
  defined as in \eqref{q1q3:lem:equivalence}. In addition, let us set
  $q_2\eqq -\waterh^2\DIV\bv + \tfrac32 q_3$.  Using the mass
  conservation equation, these definitions imply that
\begin{align*}
\partial_t q_1 + \DIV (\bv q_1) &= \partial_t \waterh^2 + \DIV(\bv\waterh^2)
= \waterh \partial_t \waterh + \waterh \bv\ADV\waterh=-\waterh^2\DIV\bv =  q_2 -\tfrac32 q_3.
\end{align*}
Hence \eqref{q1:lem:equivalence} holds true.\\
\textup{(ii)} Let us define
$\ts \eqq \partial_t q_3 + \DIV (\bv q_3)$
 and $s\eqq -\waterh\ddot{\waterh} - \tfrac32 \ts$.  Using again mass
conservation and the identity
$q_2= \waterh \dot \waterh + \tfrac32 q_3$ (recall that $\dot\waterh = -\waterh\DIV\bv$), we infer that
\begin{align*}
\partial_t q_2 + \DIV (\bv q_2) & =
\waterh (\partial_t \dot\waterh + \bv\ADV\dot\waterh) +\tfrac32( \partial_t q_3 + \DIV(\bv q_3) )
= \waterh\ddot\waterh +\tfrac32 \ts=-s.
\end{align*}
This shows that \eqref{q2:lem:equivalence} holds true. Notice also
that \eqref{q3:lem:equivalence} holds true as well since this is the definition of $\ts$. \\
\textup{(iii)}
Recalling that we have defined $q_3\eqq \waterh \bv\ADV z$, the definition of
$\dot{\sfk}$ in \eqref{SGN_topography} gives $\waterh\dot{\sfk}=\partial_t q_3
+\DIV(\bv q_3) =\ts$.
Then, using that $s\eqq -\waterh\ddot{\waterh} - \tfrac32 \ts$,
the pressure defined in \eqref{SGN_pressure} is rewritten as
follows:
\begin{align*}
p(\bu)
&= \tfrac12 g\waterh^2 +  \waterh(\tfrac13 \waterh\ddot{\waterh} + \tfrac12\waterh\dot\sfk)
= \tfrac12 g\waterh^2 +  \waterh(\tfrac13 \waterh\ddot{\waterh} +\tfrac12\ts) \\
&= \tfrac12 g\waterh^2 +  \tfrac13\waterh (\waterh\ddot{\waterh} +\tfrac32\ts)
 =\tfrac12 g\waterh^2 -\tfrac13 \waterh s.
\end{align*}
This is exactly the form of the pressure term in \eqref{q:lem:equivalence}.
Moreover,  using that $\waterh\dot{\sfk}=\ts$ and $\waterh\ddot{\waterh} = -s -\tfrac32\ts$,
the source term induced by the topography in \eqref{SGN_topography}, $r(\bu)$, becomes
\begin{align*}
  r(\bu) = g \waterh + \waterh (\tfrac12 \ddot{\waterh} + \dot{\sfk})
  = g \waterh -\tfrac12s -\tfrac34\ts +\ts =  g \waterh -\tfrac12s +\tfrac14\ts.
\end{align*}
\ie we obtain
$r(\bu)\GRAD z = (g\waterh -\tfrac12 s +\tfrac14 \ts)\GRAD z$, which is
exactly the form of the source term in \eqref{q:lem:equivalence}. In
conclusion, we have established that \eqref{mass:lem:equivalence} to
\eqref{q1q3:lem:equivalence} hold true if $\bu$ solves
\eqref{shallow_equations}--\eqref{pressure_topography_SGN}.

We now prove the converse.  Let us assume that $(\bu,q_1,q_2,q_3)$ solves
\eqref{mass:lem:equivalence}--\eqref{q1q3:lem:equivalence}. \\
\textup{(iv)}
Again we set $\bu\eqq(\waterh,\bq)$ and $\bv\eqq \waterh^{-1} \bq$. Let us set
$\dot\sfk \eqq \partial_t(\bv\ADV z) + \bv \ADV(\bv\ADV z)$. Then
using~\eqref{q3:lem:equivalence} and \eqref{q1q3:lem:equivalence} we
obtain $\waterh \dot\sfk = \ts$. Similarly, mass conservation implies that
$\waterh \dot\waterh = \partial_t(\waterh^2) + \DIV(\bv\waterh^2)$.
This identity, together with $q_1\eqq \waterh^2$ and~\eqref{q1:lem:equivalence}, gives
$\waterh \dot\waterh = q_2 -\tfrac32 q_3$. Then
\[
\waterh\ddot\waterh
= \partial_t(\waterh \dot\waterh) + \DIV(\bv\waterh\dot\waterh)
= \partial_t(q_2-\tfrac32 q_3) + \DIV(\bv(q_2-\tfrac32 q_3)) = -s -\tfrac32\ts.
\]
This implies that $s=-\waterh (\ddot\waterh +\tfrac32 \dot\sfk)$. Let
us set $p\eqq\tfrac12 g\waterh^2-\tfrac13 \waterh s$.  Then
$p=\tfrac12 g\waterh^2+ \waterh^2 (\tfrac13\ddot\waterh +\tfrac12
\dot\sfk)$.
This is the expression of the pressure
in~\eqref{SGN_pressure}. \\
\textup{(v)}
Finally let us set
$r\eqq g\waterh -\tfrac12 s +\tfrac14\ts$.  Then the above computations
imply that
$r= g\waterh +\tfrac12 \waterh (\ddot\waterh +\tfrac32 \dot\sfk) +
\tfrac14\waterh\dot\sfk$, \ie $r= g\waterh +\tfrac12 \waterh\ddot\waterh + \dot\sfk$. This is the expression
of the source term $r$ in~\eqref{SGN_topography}. Hence we have established that $\bu$ solves
\eqref{shallow_equations}--\eqref{pressure_topography_SGN}. This completes the proof.
\end{proof}

The following is another way to reformulate Lemma~\ref{lem:equivalence}, which we exploit in the next section:
\begin{proposition}[Co-dimension 2]
\label{prop:equivalence}
Let $\bu:\Dom\CROSS (0,T)\to \Real_+\CROSS \Real^d$ be a smooth
function.  Then $\bu$ solves the dispersive Serre model
\eqref{shallow_equations}--\eqref{pressure_topography_SGN}
if and only if $(\bu,q_1,q_2,q_3)$ solves the quasilinear first-order system
\eqref{mass:lem:equivalence}--\eqref{q3:lem:equivalence} on the co-dimension 2 manifold
$\{(\waterh,\bq,q_1,q_2,q_3)\in \Real_+\CROSS
\Real^d\CROSS\Real_+\CROSS\Real^2 \st q_1=\waterh^2,\ q_3=\bq\ADV
z\}$.
\end{proposition}

\begin{remark}[Initial conditions] Let $(\waterh_0,\bq_0)$ be the initial state for
\eqref{shallow_equations}--\eqref{pressure_topography_SGN}. Then the corresponding
initial state for
\eqref{mass:lem:equivalence}--\eqref{q1q3:lem:equivalence} is
\begin{align*}
&\waterh(\bx,0)=\waterh_0(\bx),\qquad\quad
\bq(\bx,0)=\bq_0(\bx), \qquad\quad
q_1(\bx,0) = \waterh_0(\bx)^2,\\
&q_3(\bx,0) = \bq_0(\bx)\SCAL\GRAD z,\quad
q_2(\bx,0) = -\waterh_0(\bx)\DIV\bv_0(\bx) +\tfrac32 q_3(\bx,0).
\end{align*}%
\end{remark}%
%%%
\subsection{The relaxation}
In this section we propose a technique to relax the constraints
$\{q_1 = \waterh^2; \ q_3=\bq\ADV z\}$ in \eqref{mass:lem:equivalence}--\eqref{q3:lem:equivalence} so that the relaxed system becomes hyperbolic and remains compatible with dry states. This is
done by adapting arguments introduced in \citep{Gu_Po_To_Ke_FAKE_SGN_2019} and initially based
on ideas from \citep{Favrie_Gavrilyuk_2017}.

We use the same notation as in \citep{Favrie_Gavrilyuk_2017} and
\citep{Gu_Po_To_Ke_FAKE_SGN_2019}. The state variable is denoted $\bu\eqq(\waterh,\bq,q_1,q_2,q_3)$. We set
$\eta\eqq \frac{q_1}{\waterh}$, $\omega\eqq \frac{q_2}{\waterh}$, and
we introduce the new variable $\beta\eqq\frac{q_3}{\waterh}$. Let
$\epsilon$ be a small length scale which we will later think of as
being some mesh-size when the model is approximated in space. Let
$\olambda$ be a non-dimensional number of order unit (say for instance
$\olambda=1$).

Let $\Gamma\in C^2(\Real;[0,\infty))$ be a smooth non-negative
function such that $\Gamma(1) = 0$ and $\Gamma'(1)=0$.  As in
\citep{Gu_Po_To_Ke_FAKE_SGN_2019} we replace $s$ in
\eqref{q2:lem:equivalence} by
$\overline{\lambda} g\frac{\waterh^2}{\epsilon}
\Gamma'(\frac{\eta}{\waterh})$.
Notice that $g\frac{\waterh^2}{\epsilon}$ scales like the square of a
velocity, which is what one should expect since
$-\overline{\lambda} g\frac{\waterh^2}{\epsilon}
\Gamma'(\frac{\eta}{\waterh})$
should be an ansatz for
$-\waterh(\ddot\waterh+\tfrac32\dot{\sfk})$. The purpose of this term
is to enforce the ratio $\frac{\eta}{\waterh}$ to be close to $1$ (\ie
$q_1\rightarrow\waterh^2$ as $\epsilon \to 0$).

Let $\Phi\in C^0(\Real;\Real)$ be a function such that
$\xi\Phi(\xi) \ge 0$ for all $\xi\in \Real$. Let
$\waterh_0$ be a reference water height.  We are going to replace
$\ts$ in \eqref{q3:lem:equivalence} by
$\overline{\lambda}g\waterh_0
\frac{\waterh}{\epsilon}\Phi((\bv\cdot\GRAD z
-\beta)/\sqrt{g\waterh_0})$.
The purpose of this term is to enforce
$(\bv\SCAL\GRAD z -\beta)/\sqrt{g\waterh_0}$ to be close to $0$ (\ie
$q_3\rightarrow\bq\cdot\GRAD z$ as $\epsilon \to 0$).

Recalling that $\bv\eqq\waterh^{-1}\bq$, $q_1\eqq\waterh \eta$,
$q_2 \eqq \waterh \omega $ and $q_3\eqq\waterh\beta$, the relaxed system we
consider in the rest of the paper is formulated as follows:
\begin{subequations}\label{full_relaxed}
\begin{align}
&\partial_t \waterh + \DIV\bq = 0,\label{mass:relax}\\\
&\partial_t \bq + \DIV(\bv\otimes \bq) + \GRAD p_\epsilon(\bu)
=-r_\epsilon(\bu) \GRAD z,\label{q:relax}\\
&\partial_t q_1 + \DIV (\bv q_1) = q_2 -\tfrac32 \bq\ADV z,\label{q1:relax}\\
&\partial_t q_2 + \DIV (\bv q_2) = -s_\epsilon(\bu), \label{q2:relax}\\
&\partial_t q_3 + \DIV (\bv q_3) = \ts_\epsilon(\bu), \label{q3:relax} \\
&p_\epsilon(\bu) \eqq \tfrac12 g\waterh^2 + \tp_\epsilon(\bu),\qquad
\tp_\epsilon(\bu) \eqq -\tfrac13 \tfrac{\overline{\lambda}g}{\epsilon}
\waterh^2\big(\eta\Gamma'(\tfrac{\eta}{\waterh})-2\waterh \Gamma(\tfrac{\eta}{\waterh})\big),
\label{pressure_relax} \\
& r_\epsilon(\bu) \eqq g\waterh -\tfrac12 s_\epsilon(\bu) + \tfrac14\ts_\epsilon(\bu), \label{source_relax} \\
&s_\epsilon(\bu) \eqq \overline{\lambda} g\tfrac{\waterh^2}{\epsilon}
\Gamma'(\tfrac{\eta}{\waterh}),\qquad\qquad
\ts_\epsilon(\bu) \eqq \overline{\lambda}g\waterh_0 \tfrac{\waterh}{\epsilon}\Phi(\tfrac{\bv\GRAD z
-\beta}{\sqrt{g\waterh_0}}). \label{penalty_relax}
\end{align}
\end{subequations}

\begin{remark}[$q_3$]\label{Rem:q3}
Notice that we have replaced $-\frac32 q_3$ on the right-hand side of \eqref{q1:relax}
by $-\tfrac32 \bq\ADV z$.  This is consistent since  $q_3$ should be equal to $\bq\ADV z$. This
change is justified by the energy argument in Lemma~\ref{Lem:relax_energy}.
\end{remark}

\begin{remark}[Definition of $\Phi$]\label{Rem:def_of_Phi} In the applications
    reported at the end of the paper we take $\Phi(\xi)=\xi$, but we
    prefer to present the method with a generic function $\Phi$ to
    emphasize the generality of the relaxation procedure.
\end{remark}

\begin{remark}[Pressure] \label{Rem:pressure}%
The expression for
    the pressure $\tp_\epsilon(\bu)$ in \eqref{pressure_relax} is
    fully justified by the energy argument in
    Lemma~\ref{Lem:relax_energy}. After replacing $s$ by
  $\overline{\lambda} g\frac{\waterh^2}{\epsilon}
  \Gamma'(\frac{\eta}{\waterh})$
  in \eqref{q:lem:equivalence}, we observe that the definition of
  $\tp_\epsilon(\bu)$ is compatible with the definition
  $\tp(\bu) = -\frac13 \waterh s$ up to the remainder
  $\frac13 \frac{\overline{\lambda}g}{\epsilon} 2 \waterh^3
  \Gamma(\tfrac{\eta}{\waterh})$,
  which indicates that \eqref{full_relaxed} may not be consistent
    with \eqref{eq:lem:equivalence}. This is not the case since this
  remainder is small when the ratio $\frac{\eta}{\waterh}$ is close to
  $1$. More precisely, using Taylor expansions at $1$, we have
\begin{align*}
\Gamma(1)=0=\Gamma(\frac{\eta}{\waterh})+ \waterh^{-1} (\waterh -\eta) \Gamma'(\frac{\eta}{\waterh})
+  \waterh^{-2} \calO(\waterh -\eta)^2,
\end{align*}
which shows that $2\waterh\Gamma(\tfrac{\eta}{\waterh})/
\eta |\Gamma'(\frac{\eta}{\waterh})| = \calO(\tfrac{|\eta-\waterh|}{\eta})$.
Hence the ratio $2\waterh\Gamma(\tfrac{\eta}{\waterh})/
\eta |\Gamma'(\frac{\eta}{\waterh})|$ is small as
$\eta\to \waterh$, which proves that $\tp_\epsilon(\bu)$ is indeed a consistent approximation
of $\tp(\bu) = -\frac13 \waterh s$ as $\eta\to \waterh$.
\end{remark}
\begin{remark}[Comparisons with
    \citep{Gu_Po_To_Ke_FAKE_SGN_2019}] The incomplete system
    considered in \citep{Gu_Po_To_Ke_FAKE_SGN_2019} consists of solving
    for $(\waterh,\bq,q_1,q_2)\tr$ so that
\begin{subequations}\label{inc_relaxed}
\begin{align}
&\partial_t \waterh + \DIV\bq = 0,\\
&\partial_t \bq + \DIV(\bv\otimes \bq) + \GRAD (\tfrac12 g\waterh^2 + \tp_\epsilon(\bu))
= -g\waterh \GRAD z, \label{inc_relaxed_mt}\\
&\partial_t q_1 + \DIV (\bv q_1) = q_2, \label{inc_relaxed_q1}\\
&\partial_t q_2 + \DIV (\bv q_2) = -s_\epsilon(\bu), \\
&\tp_\epsilon(\bu) \eqq -\tfrac13 \tfrac{\overline{\lambda}g}{\epsilon}
\waterh^2\big(\eta\Gamma'(\tfrac{\eta}{\waterh})-2\waterh \Gamma(\tfrac{\eta}{\waterh})\big),
\qquad s_\epsilon(\bu) \eqq \overline{\lambda} g\tfrac{\waterh^2}{\epsilon}
\Gamma'(\tfrac{\eta}{\waterh})\label{inc_press_and_source}.
\end{align}
\end{subequations}
Notice that the expression for the relaxed non-hydrostatic pressure
\eqref{inc_press_and_source} is the same as in~\eqref{full_relaxed}.
But, the relaxed pressure \eqref{inc_press_and_source} for the
incomplete system only approximates $\frac13 \waterh^2\ddot \waterh$,
whereas the relaxed pressure in \eqref{full_relaxed} approximates the
$\waterh^2(\frac13\ddot \waterh +\frac12\dot \sfk)$. The right-hand
sides of \eqref{inc_relaxed_mt} and \eqref{inc_relaxed_q1} are also
different. We also note that the incomplete system \eqref{inc_relaxed}
does not contain an evolution equation for the quantity $q_3$ which is
an \textit{ansatz} for $\bq\SCAL\GRAD z$ as in \eqref{full_relaxed}.
Thus, although the two relaxation techniques bear resemblance, they
are significantly different. This difference is numerically
illustrated in Section~\ref{Sec:santos_triangle}. Let us note though
that the two systems are equivalent when the topography is trivial.
\end{remark}
The following result is the relaxed counterpart of Lemma~\ref{lem:energy}.
\begin{lemma}\label{Lem:relax_energy}
Let $\bu$ be a smooth solution to \eqref{mass:relax}--\eqref{penalty_relax}. Then
the following holds true:
$\partial_t \calE_\epsilon(\bu) + \DIV(\bcalF_\epsilon(\bu)) =\tfrac14 \ts_\epsilon(\bu) (\beta - \bv\ADV z) \le 0$,
with
\begin{subequations}
\begin{align}
\calE_\epsilon(\bu) &\eqq \tfrac12 g \bal{(\waterh +  z)^2}  + \tfrac12 \waterh \bv^2
+ \tfrac16\waterh \omega^2 +  \tfrac18 \waterh \beta^2
+\tfrac{\overline{\lambda}g}{3\epsilon}\waterh^3\Gamma(\tfrac{\eta}{\waterh}),\label{energy_relax} \\
\bcalF_\epsilon(\bu) &\eqq \bv (\calE_\epsilon(\bu) \bal{-\tfrac12 g z^2} + p_\epsilon(\bu)).
\end{align}
\end{subequations}
\end{lemma}
\begin{proof}
\textup{(i)} The first part of the argument is standard.
We multiply the mass conservation equation by $g (\waterh+z)$
  and the momentum equation by $\bv$, use the mass conservation
  equation, and add the results:
\begin{multline*}
\partial_t (\tfrac12 g \waterh^2 +gz\waterh + \tfrac12 \waterh \bv^2) + \DIV ( \bv(\tfrac12 g \waterh^2 +
g z\waterh + \tfrac12 \waterh\bv^2+\tfrac12 g \waterh^2))
+\bv\ADV p_\epsilon(\bu) \\
= \tfrac12 s_\epsilon(\bu) \bv\ADV z - \tfrac14\ts_\epsilon(\bu) \bv\ADV z.
\end{multline*}
\textup{(ii)} We now multiply  \eqref{q1:relax}
by $\tfrac{\overline{\lambda}g}{3\epsilon} \waterh \Gamma'(\tfrac{\eta}{\waterh})
\eqq \frac{1}{3\waterh}s_\epsilon(\bu)$ and proceed as in the proof of Proposition~3.1
in \citep{Gu_Po_To_Ke_FAKE_SGN_2019}. Recalling that $\epsilon$ is constant and using
the definition of $\tp(\waterh,\eta)$, we obtain
\[
\partial_t (\tfrac{\overline{\lambda}g}{3\epsilon}\waterh^3\Gamma(\tfrac{\eta}{\waterh}))
+ \DIV(\tfrac{\overline{\lambda}g}{3\epsilon}\waterh^3\Gamma(\tfrac{\eta}{\waterh})\bv)
+ \tp(\waterh,\eta) \DIV \bv
= \tfrac13 s_\epsilon(\bu) \omega - \tfrac12 s_\epsilon(\bu) \bv\ADV z.
\]
Notice that, as mentioned in Remark~\ref{Rem:q3}, replacing
$-\frac32 q_3$ by $-\frac32 \waterh \bv\ADV z$ in \eqref{q1:relax} is important
here. Without this substitution we would have
$-\tfrac12 s_\epsilon(\bu)\beta$ on the right-hand side of the above identity instead of
$- \tfrac12 s_\epsilon(\bu) \bv\ADV z$. \\
\textup{(iii)} We continue by multiplying \eqref{q2:relax} by $\frac13\omega$ and we obtain
\[
\partial_t (\tfrac16\waterh \omega^2) + \DIV(\tfrac16\waterh \omega^2\bv)
=-\tfrac{1}{3} s_\epsilon(\bu) \omega.
\]
\textup{(iv)} Finally we multiply \eqref{q3:relax} by $\frac14\beta$ and we obtain
\[
\partial_t (\tfrac18\waterh \beta^2) + \DIV(\tfrac18\waterh \beta^2\bv)
=\tfrac{1}{4} \ts_\epsilon(\bu) \beta.
\]
\textup{(iv)}
We conclude by adding the four identities obtained above, and
we obtain $\partial_t \calE_\epsilon(\bu) + \DIV(\bcalF_\epsilon(\bu)) =\tfrac14 \ts_\epsilon(\bu) (\beta - \bv\ADV z)$.
Notice that the definition of $\ts_\epsilon(\bu)$ and $\Phi(\xi)$ gives
\[
\tfrac{1}{4} \ts_\epsilon(\bu)(\beta -\bv\ADV z)
=  -\tfrac{g\waterh_0}{4} \tfrac{\waterh}{\epsilon}\Phi(\tfrac{\bv\GRAD z
-\beta}{\sqrt{g\waterh_0}})(\bv\ADV z-\beta) \le 0.
\]
This completes the proof.
\end{proof}
\bal{
\begin{corollary}
  Let $\bu\eqq(\waterh,\bq,q_1, q_2, q_3)\tr$ be a smooth solution to
  \eqref{mass:relax}--\eqref{penalty_relax} and assume $\epsilon$ is
  constant (\ie does not depend on $\bx$ and $t$). Let $T > 0$ be some
  final time. Assume that the boundary conditions for $\bu$ are such
  that $\bcalF(\bu)\SCAL\bn_{\partial D} = 0$ for all
  $t\in(0,T)$. Then, there is a $c(\bu_0)$ such that:
\begin{equation}
  \int_D\left(\waterh^3\Gamma(\tfrac{\eta}{\waterh})\right)_{|t = T} \leq c(\bu_0)\epsilon
\end{equation}
\end{corollary}
}
\begin{remark}[\eqref{energy_relax} vs. \eqref{DSV_energy}]
  By comparing the expression~\eqref{energy_relax} to \eqref{DSV_energy}, and
  recalling that $\omega$ is meant to be an approximation for
  $\dot{\waterh} + \frac32\bv\ADV z$ and $\beta$ an approximation for
  $\bv\ADV z$, we see that
  $\tfrac16\waterh \omega^2 + \tfrac18 \waterh \beta^2 =
  \tfrac16\waterh(\omega^2 + \tfrac34\beta^2)$
  in \eqref{energy_relax} is the approximation of
  $\tfrac16 \waterh \Big(\big(\dot{\waterh} + \tfrac32 (\bv\ADV
  z)\big)^2 +\tfrac34 (\bv\ADV z)^2 \Big)$
  in \eqref{DSV_energy}. As in \citep{Gu_Po_To_Ke_FAKE_SGN_2019}, we
  also observe the extra term in the energy
  $\waterh^3\Gamma(\tfrac{\eta}{\waterh})$, but this was shown in Remark~\ref{Rem:pressure}
  to be a small, positive quantity.
\end{remark}
% for reviewer #1
\begin{remark}[Energy balance]
Without topography, the statement in Lemma~\ref{Lem:relax_energy} gives an exact energy balance
as in \cite{Favrie_Gavrilyuk_2017} and \citep{Gu_Po_To_Ke_FAKE_SGN_2019}.
\end{remark}
%%%
\subsection{Alternative reformulations of the dispersive Serre model}\label{Sec:alternative}
There are many ways to reformulate the dispersive Serre model into a system of first-order
conservation equations with sources. For instance, in \cite{Gavrilyuk_1996} the authors
  reformulated the model with a flat bottom as a first-order system
  with a constraint on the divergence of the velocity (equations
  (5.12)--(5.15) therein). The model from \citep{Gavrilyuk_1996} is also used in
\cite{Bristeau_Mangenay_SaniteMarie_Seguin_2015} (equation (50) therein). Then,
 following \citep{Gavrilyuk_1996} and
\citep{Bristeau_Mangenay_SaniteMarie_Seguin_2015}, \cite{Escalante_Dumbser_Castro_2019} included some
effects of the topography into this first-order system with the
assumption that the topography was mildly varying (\ie dropping the
terms containing $\DIV(\GRAD z)$ and $\|\GRAD z\|^2$). They also relaxed the
system to enforce hyperbolicity. To put the present work in
perspective with respect to these techniques, we recall the
reformulations proposed in
\citep{Gavrilyuk_1996,Bristeau_Mangenay_SaniteMarie_Seguin_2015,Escalante_Dumbser_Castro_2019}
(but contrary to \citep{Escalante_Dumbser_Castro_2019}
we keep all the effects induced by the topography).

The starting point is again the system
\eqref{shallow_equations}--\eqref{pressure_topography_SGN} with the
topography effects
\bal{from~\cite{green_naghdi_1976,seabra-santos_renouard_temperville_1987}.}  Let
$w \eqq -\tfrac12\waterh\DIV\bv + \tfrac34\bv\ADV z$. Then, using that
$\dot{\waterh} = -\waterh\DIV\bv$ and using the notation as above
$\dot\sfk \eqq \partial_t(\bv\ADV z) + \bv\ADV(\bv\ADV z)$, we have
that
\[
D_t w:= \partial_t w + \bv\ADV w = \frac12\ddot{\waterh} + \frac34\dot\sfk
= \frac32\Big(\frac13\ddot{\waterh} + \frac12\dot\sfk\Big)
= \frac32 \frac{\overline{p}(\bu)}{\waterh},
\]
where
$\overline{p}(\bu):=\tfrac13\waterh\ddot{\waterh} +
\tfrac12\waterh\dot\sfk$.
Then, using mass conservation, the above can be written as
$\partial_t(\waterh w) + \DIV(\bu \waterh w) =
\tfrac32\overline{p}(\bu)$.
Combining everything, another reformulation of the dispersive
Serre model with topography is given as follows:
\begin{subequations} \label{SainteMarie_etal_Castro}
\begin{align}
&\partial_t \waterh + \DIV\bq = 0,\\
&\partial_t \bq + \DIV(\bv\otimes \bq) + \GRAD (\tfrac12g\waterh^2 + \waterh \overline{p}(\bu))
=-(g \waterh + \tfrac32\overline{p}(\bu) + \tfrac14\waterh\dot\sfk)\GRAD z,\\
&\partial_t (\waterh w) + \DIV (\bu \waterh w) = \tfrac32\overline{p}(\bu),\\
&\DIV\bv + \frac{w - \tfrac34\bv\cdot\GRAD z}{\frac12\waterh} = 0. \label{SainteMarie_etal_Castro_constraint}
\end{align}
\end{subequations}
We recover Eq.~(1) in \citep{Escalante_Dumbser_Castro_2019},
up to the term $\tfrac14\waterh\dot\sfk$, which is neglected therein,
and up to the coefficient $\frac34$ in \eqref{SainteMarie_etal_Castro_constraint}.
The above system bears some resemblance
to \eqref{eq:lem:equivalence}. In particular we observe that $\waterh \overline{p}(\bu)=-\frac13 s(\bu)$ and
 $\waterh w = \frac12 q_2$. Notice however that
in our system~\eqref{eq:lem:equivalence} the two constraints~\eqref{q1q3:lem:equivalence}
are purely algebraic (\ie these constraints are enforced in the phase space),
whereas the constraint \eqref{SainteMarie_etal_Castro_constraint} is differential.
As a result, the technique proposed in \citep{Escalante_Dumbser_Castro_2019} to relax
the differential constraint \eqref{SainteMarie_etal_Castro_constraint}
is fundamentally different from \eqref{penalty_relax}.
\bal{
\begin{remark}[Reformulation in \cite{nieto_parisot_2018}]
  In~\cite{nieto_parisot_2018}, the authors reformulate the Serre
  model with full topography effects.  The authors introduce two
  constraints for the reformulation:
  $w_s\eqq -\waterh\DIV\bv + \bv\SCAL\GRAD z$ and
  $\tilde{w}\eqq w_s + \frac12\waterh\DIV\bv$. Combining these two
  constraints into one yields:
  $\tilde{w}\eqq - \frac12\waterh\DIV\bv + \bv\SCAL\GRAD z.$ This
  constraint is similar, up to the constant on the $\bv\SCAL\GRAD z$
  term, to the constraint defined in \S\ref{Sec:alternative}. Thus,
  one can repeat the above process and derive the first-order
  formulation introduced in~\citep{nieto_parisot_2018}. Notice that
  in~\citep{nieto_parisot_2018}, the constraint is again differential
  and thus different from the proposed reformulation in this work.
\end{remark}
}
%%%
\section{Numerical Illustrations}
\label{Sec:numerical_illustrations}
In this section we illustrate the performance of the relaxation algorithm~\eqref{full_relaxed}.

\subsection{Numerical details} \label{Sec:Numer_details}
We use a continuous finite element technique similar to that
described
in~\citep{Guermond_Quezada_Popov_Kees_Farthing_2018,Gu_Po_To_Ke_FAKE_SGN_2019}.
The finite elements are piecewise linear and the time stepping is done with
the third-order, three step, strong stability preserving Runge Kutta technique (SSP RK(3,3)).
We set $\Gamma(x) = 3(x-1)^2$ and $\Phi(\xi) = \xi$ in
\eqref{penalty_relax}.  We also take $\overline\lambda = 1$ and
$\epsilon$ is the local mesh-size.  Denoting by
$\{\varphi_i\}_{i\in\calV}$ the global shape functions, the local
mesh-size is defined to be
$\epsilon_i\eqq (\int_{\Dom} \varphi_i \diff x)^{\frac1d}$, where
$d\in\{1,2\}$ is the space dimension.  The numerical viscosity is
the entropy viscosity defined in
\citep[\S{6}]{Guermond_Quezada_Popov_Kees_Farthing_2018}. The
estimate of the maximum wave speed in the elementary Riemann
problems is detailed in~\S{4.2} therein. The method is positivity
preserving and well-balanced. The detailed implementation of the
method is reported in~\cite{Tovar_PhD}. \bal{The boundary conditions for each
numerical illustration are detailed in the respective subsections. We consider
either Dirichlet or wall boundary conditions. Though it is possible to construct
a treatment to enforce outflow boundary conditions, we note that at the moment
it is not immediately obvious how to solve the associated Riemann Problem for the
hyperbolic relaxed system~\eqref{full_relaxed}.}

\subsection{Accuracy with fixed $\epsilon$} In a first series of
  tests, not reported here for brevity, we estimate the accuracy of
  the method by fixing the relaxation parameter $\epsilon$, \ie
  $\epsilon$ does not depend on the mesh-size.  These tests show that
  the method gives an approximation of the solution to
  \eqref{full_relaxed} that is third-order accurate in time and
  second-order accurate in space. It is common in the literature to
  fix the relaxation parameter to estimate the accuracy of the space and time approximation;
  we refer the reader for instance to
  \cite[Tab.~1]{Escalante_Dumbser_Castro_2019}, \cite[Tab.~1]{dumbser_2020},
\cite[Fig.~6]{Favrie_Gavrilyuk_2017}

%%%
\subsection{Steady state solution}
We now demonstrate the accuracy of the relaxation technique using
  the steady state solution described in
  \S\ref{Sec:steady_state_solution} with $\epsilon$ depending on the
  mesh-size as explained in \S\ref{Sec:Numer_details}.  The
bathymetry and the water height are given in
Lemma~\ref{Lem:steady_state_solution}.  We set $\waterh_0=\SI{1}{m}$,
$a=0.2$, $g=\SI{9.81}{m s^{-2}}$. This gives a discharge value of
  $q=\sqrt{5.886}\SI{}{m^2s^{-1}}$ and coefficient
  $r=\sqrt{0.5}\SI{}{m^{-1}}$ given by the expressions
  in~\eqref{eq:Lem:steady_state_solution}.  The simulations are done
with $\Dom=(-\SI{10}{m},\SI{15}{m})$. The discharge is enforced at the
inflow boundary $x=-\SI{10}{m}$. The water height is enforced at
$x=-\SI{10}{m}$ and $x=\SI{15}{m}$.  We show in
Table~\ref{Tab:steady_state} the numerical results obtained at
$t=\SI{1000}{s}$. The number of grid points is shown in the leftmost
column. The relative errors on the water height measured in the
$L^1$-norm, $E_1\eqq\|\waterh-\waterh_h\|_{L^1}\!/\|\waterh\|_{L^1} $,
and the relative errors measured in the $L^\infty$-norm,
$E_2\eqq\|\waterh-\waterh_h\|_{L^\infty}\!/\|\waterh\|_{L^\infty}$,
are shown in the second and third columns. The relative $L^1$-norm of
the difference between $\waterh_h^2$ and $q_{1h}$,
$E_3\eqq \|\waterh_h^2-q_{1h}\|_{L^1}/\|q_{1h}\|_{L^1}$, and the
relative $L^1$-norm of the difference between $q_h\partial_x z$ and
$q_{3h}$,
$E_4\eqq\|q_h\partial_x z - q_{3h}\|_{L^1}/\|q_h\partial_x z\|_{L^1}$,
are shown in the fourth and fifth columns.  We observe that all
  the quantities converge with a first-order rate with respect to the
  mesh-size. This is consistent since we chose the relaxation
  parameter in the relaxed system~\eqref{full_relaxed} to be
  proportional to the local mesh-size which gives a first-order
  approximation of the fully coupled
  system~\eqref{eq:lem:equivalence}. Similar convergence results are
  reported in \cite[Tab.~2]{dumbser_2020}.
\begin{table}[h!]\small
\centering
\begin{tabular}{|c| c c| c c |c c| c c|} \hline
\Nglob
& \multicolumn{2}{c|}{$E_1$}
& \multicolumn{2}{c|}{$E_2$}
& \multicolumn{2}{c|}{$E_3$}
& \multicolumn{2}{c|}{$E_4$}\\
 \hline
 100 & 1.98E-03 & Rate & 7.55E-03 & rate & 8.54E-05 & Rate & 1.33E-01 & Rate \\ \hline
 200 & 1.09E-03 & 0.86 & 3.15E-03 & 1.26 & 3.83E-05 & 1.16 & 6.77E-02 & 0.97 \\ \hline
 400 & 4.23E-04 & 1.36 & 1.05E-03 & 1.58 & 1.76E-05 & 1.12 & 3.40E-02 & 0.99 \\ \hline
 800 & 1.73E-04 & 1.29 & 4.07E-04 & 1.37 & 8.51E-06 & 1.05 & 1.70E-02 & 1.00 \\ \hline
1600 & 7.92E-05 & 1.13 & 1.82E-04 & 1.16 & 4.20E-06 & 1.02 & 8.51E-03 & 1.00 \\ \hline
3200 & 3.62E-05 & 1.13 & 8.51E-05 & 1.10 & 2.09E-06 & 1.01 & 4.25E-03 & 1.00 \\ \hline
6400 & 1.85E-05 & 0.97 & 4.31E-05 & 0.98 & 1.05E-06 & 1.00 & 2.13E-03 & 1.00 \\ \hline
\end{tabular}%
\caption{Steady state solution with topography.}%
\label{Tab:steady_state}%
\end{table}

\begin{figure}[h]
\begin{subfigure}{0.49\textwidth}
\includegraphics[trim=12 32 12 41,clip=,width=\textwidth]{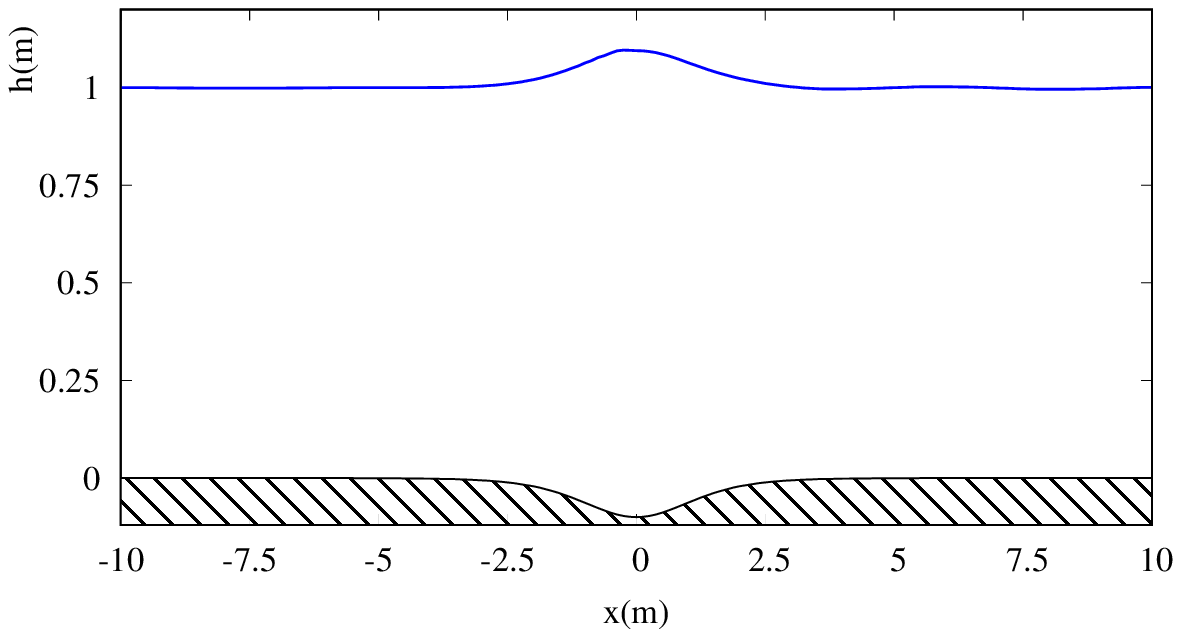}
\caption{Full system~\eqref{full_relaxed}} \label{Fg:1D_steady_state_with}
\end{subfigure} \hfill
\begin{subfigure}{0.49\textwidth}
\includegraphics[trim=12 32 12 41,clip=,width=\textwidth]{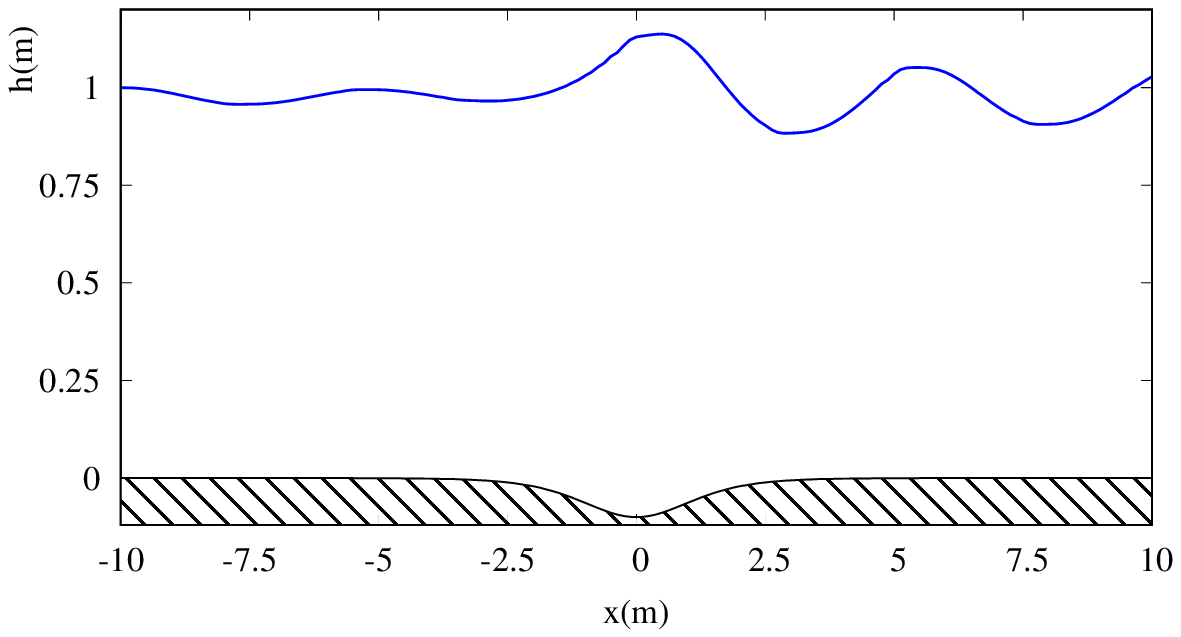}
\caption{Partial system \eqref{inc_relaxed} from \citep{Gu_Po_To_Ke_FAKE_SGN_2019}}\label{Fg:1D_steady_state_without}
\end{subfigure}
\vspace*{-\baselineskip}
\caption{Steady state solution with topography.}\label{Fg:1D_steady_state}
\vspace*{-\baselineskip}
\end{figure}
We show in Figure~\ref{Fg:1D_steady_state_with} the graph of the
steady-state solution obtained numerically by
solving~\eqref{penalty_relax} (the results are shown only over the
interval $(-10,10)$). The grid is composed of $200$ $\polP_1$ elements and the
final time is $t=\SI{1000}{s}$.  In
Figure~\ref{Fg:1D_steady_state_without} we show the graph of the
steady-state solution obtained by solving the system \eqref{inc_relaxed} from
\citep{Gu_Po_To_Ke_FAKE_SGN_2019}. The effects of the missing terms is clear. It would be interesting to
see which of these two states can be reproduced experimentally.
%%%
\subsection{1D Solitary wave propagating over a triangular obstacle}\label{Sec:1D_solitary_over_triangle}
\label{Sec:santos_triangle}
We now consider the experiments of a solitary wave propagating over a
triangular obstacle described in
\cite{seabra-santos_renouard_temperville_1987} and focus on the study
of reflected waves. The goal here is to compare
  the system \eqref{full_relaxed} to the experiments and to the
  incomplete model \eqref{inc_relaxed} introduced in
  \citep{Gu_Po_To_Ke_FAKE_SGN_2019}. We discover  that
  the prediction of the amplitude of the reflected waves  are
  exceptionally accurate when the topography corrections are fully account
  for. It seems that this observation and a thorough examination of the data provided in
  \citep{seabra-santos_renouard_temperville_1987} for the triangular obstacle had never been done
  before.

In the experiments reported in
\citep{seabra-santos_renouard_temperville_1987}, a reflected wave is
generated when the solitary wave passes over the triangular obstacle
for some values of the water depth $\waterh_0$ and incident wave
amplitude $\alpha$.
We focus on the water depths of
  $\SI{15}{cm}$ and $\SI{12.5}{cm}$ since the authors claim that the
  height of the reflected wave is negligible when the water depth is
  larger.
The bathymetry in the experiments is a triangular obstacle
centered at $x=\SI{0}{m}$ with a base of $\SI{14.1}{cm}$ and height
$\SI{10}{cm}$. This triangular obstacle can be reproduced with
  the function $z(x) = \max(0.1 - \tfrac{10}{7.05}|x|,0)$.  However,
  the original Serre
  model~\eqref{shallow_equations_def_fluxes}--\eqref{pressure_topography_SGN}
  contains terms proportional to $\LAP z(\bx)$ (and
  $\GRAD(\LAP z(\bx))$) which produce a Dirac measure and the
  derivative of a Dirac measure at the PDE level. We handle this difficulty
by introducing a mesh-dependent smoothing of the bathymetry as follows:
  $z(x) = \max(0.1 - \tfrac{10}{7.05}\frac{x^2}{\abs{x} +
    d\sqrt{\waterh_0 h}},0),$
  where $h$ is the mesh-size and $\waterh_0$ is the water depth. The
  constant $d$ is chosen so that on a mesh composed of 1600 $\polP_1$ elements,
  the smoothing parameter $d\sqrt{\waterh_0 h}$ is equal to
  $0.075$. We have numerically verified that this smoothing procedure
  is consistent in the sense that the solution converges in the
  $L^1$-norm as we refine the mesh.
%For very fine meshes, spurious
%  oscillations occur near the location of the triangular peak. This is
%  expected as our approximate smoothed bathymetry converges to the
%  true triangular obstacle.

The computational domain is set to $\Dom=(-\SI{20}{m},
\SI{20}{m})$.
We reproduce 9 of the experiments shown in
\citep[Tab.~2]{seabra-santos_renouard_temperville_1987} for
$\waterh_0=\SI{15}{cm}$ and $\waterh_0=\SI{12.5}{cm}$ with the values
of the incident wave height $\alpha$ given in Table
\ref{table:shelf_ref}. Let $\tilde{\waterh}(x,t)$ and
  $\tilde{u}(x,t)$ be the water height and velocity of an solitary
  wave:
\begin{align}
  \tilde{\waterh}(x,t) = \waterh_0+\frac{\alpha}{(\cosh(r (x - x_0-ct)))^2},\qquad
  \tilde{u}(x,t) = c\frac{\tilde{\waterh}(x,t)-\waterh_0}{\tilde{\waterh}(x,t)},\label{eq:initial_solitary}%
\end{align}%
with wave speed $c = \sqrt{g (\waterh_0 + \alpha)}$ and width
$r = \sqrt{\frac{3 \alpha}{4 \waterh_0^2(\waterh_0 + \alpha)}}$. To be consistent with the experimental measurements, we initiate the
solitary wave at $x_0 = -15\waterh_0$ with
\begin{equation}
\waterh(x,0) = \text{max}\{\tilde{\waterh}(x,0)-z(x),0\}, \qquad
q(x,0) = \tilde{u}(x,0)\waterh(x,0).\label{solitary_wave}
\end{equation}
Here we take $g=\SI{9.81}{ms^{-2}}$.
We then measure the height of the reflected
wave at the location $x = -25\waterh_0$.  We run the computations with
a mesh composed of 3200 $\polP_1$ elements with CFL 0.1 until final time
$t=\SI{10}{s}$. In Table~\ref{table:shelf_ref}, we report the results
of our computations for the system \eqref{full_relaxed} and the
incomplete system \eqref{inc_relaxed} (denoted by the index
$_{\text{inc}}$). In the table, $E_r$
is the relative difference between the computational and experiment
values (shown in \citep[Tab.~2]{seabra-santos_renouard_temperville_1987}) and is defined as $E_r=\frac{\alpha_r-\alpha_\text{Exp}}{\alpha_\text{Exp}}$ for the reflected wave. Note that the amplitude values reported in Table~\ref{table:shelf_ref} are in centimeters so a 10\% error is only $\SI{1}{mm}$. We see in the table that we have good agreement for
the full system \eqref{full_relaxed} while the incomplete system \eqref{inc_relaxed} gives
very poor agreement by largely overshooting the reflected wave
amplitudes.
%table goes Here
\begin{table}[h!]
\centering
\begin{tabular}{||c c c c c c c||}
 \hline
 Exp. & $\waterh_0$ & $\alpha$ & $\alpha_{r}$ & $E_r$ &  $\alpha_{\text{inc}, r}$ & $E_{\text{inc},r}$ \\ [0.5ex]\hline\hline
 72 & 15.0 & 2.96 & 0.37 & -9.8\% & 0.61 & 49\% \\
 73 & 15.0 & 4.35 & 0.53 & -12\% & 1.07 & 78\% \\
 74 & 15.0 & 5.81 & 0.70 & 7.7\% & 1.64 & 150\% \\
 75 & 15.0 & 6.56 & 0.79 & -1.3\% & 1.93 & 140\% \\
 76 & 15.0 & 8.40 & 0.96 & -20\% & 2.56 & 110\% \\
 77 & 12.5 & 2.50 & 0.49 & -18\% & 0.75 & 25\%\\
 78 & 12.5 & 4.75 & 0.86 & 7.5\% & 1.66 & 110\%\\
 79 & 12.5 & 6.0 & 1.03 & 8.4\% & 2.21 & 130\% \\
 80 & 12.5 & 6.3 & 1.07 & 1.9\% & 2.34 & 120\%  \\
 [1ex]
 \hline
\end{tabular}
\caption{Numerical results for the passing of a solitary wave over a triangular
  obstacle (reflected waves). $\waterh_0$ is the still water depth;
  $\alpha$ is the amplitude of the incident wave;
  $\alpha_{r}$ is the amplitude of the reflected wave;
  $\alpha_{\text{inc},r}$ is the amplitude of the reflected wave for
  incomplete system \eqref{inc_relaxed}. All values are in \SI{}{cm}.}
\label{table:shelf_ref}
\end{table}
In Figures \ref{fig:triangle_newModel} and
\ref{fig:triangle_oldModel}, we show the graph of the free surface
$\waterh + z$ in the $(x,t)$-plane for Experiment 78 listed in
Table~\ref{table:shelf_ref} (the left panel shows the results for the
full system~\eqref{full_relaxed}; the right panel shows the results
for the incomplete system \eqref{inc_relaxed}).  The
effects of the missing terms are evident: we see that the incomplete
system \eqref{inc_relaxed} generates more and higher reflected waves, some of which are likely unphysical. Likewise, the incomplete system creates a wave of large amplitude when passing over the triangular peak, which we suspect is unphysical.
% Solitary wave triangle figure %
\begin{figure}
\centering
\begin{subfigure}{0.49\textwidth}
  \centering
    \includegraphics[trim={90 0 90 10},clip,width=\textwidth]{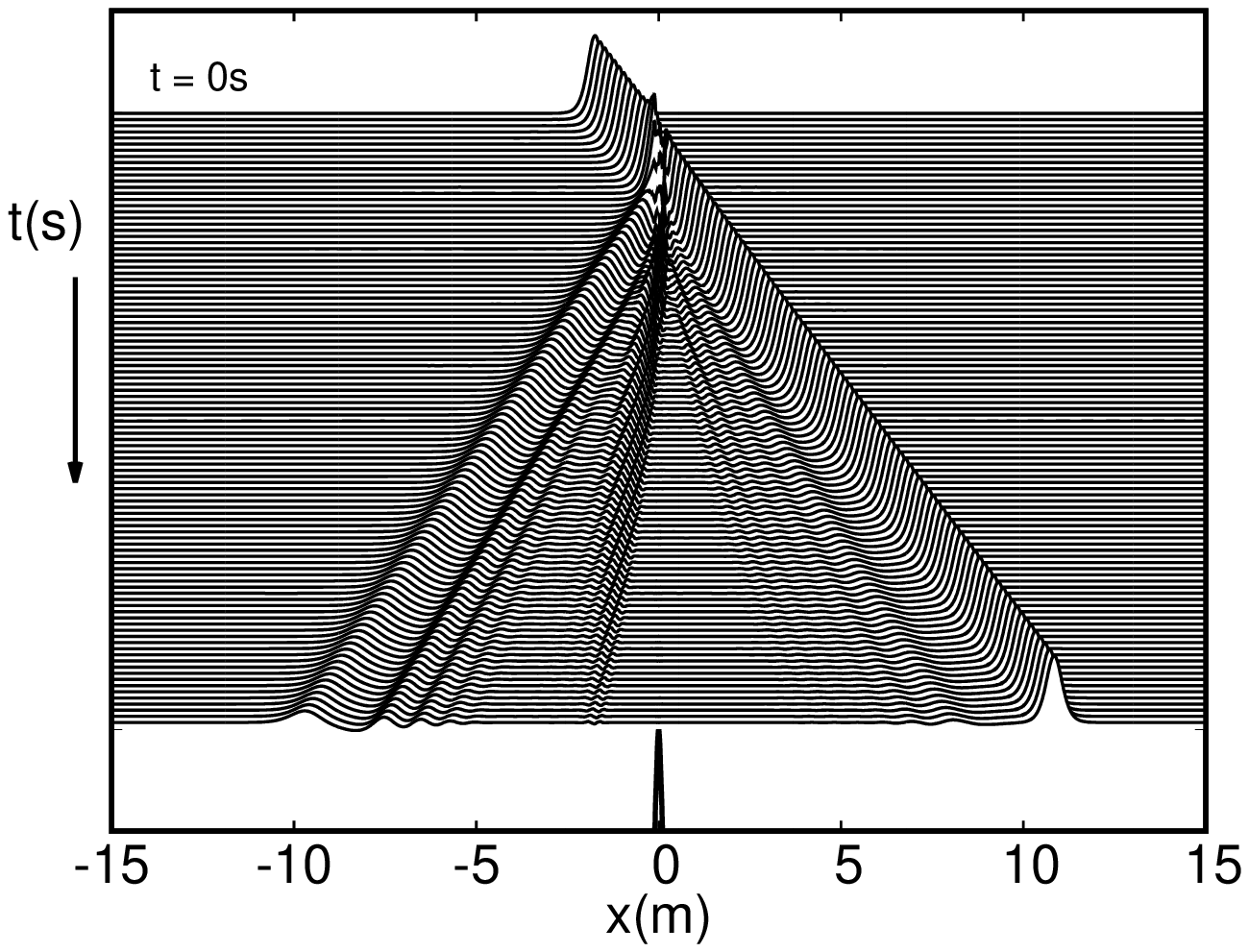}
    \caption{$(x,t)$-plane representation of $\waterh+z$
for the relaxed system \eqref{full_relaxed}.}\label{fig:triangle_newModel}%
\end{subfigure} \hfill
\begin{subfigure}{0.49\textwidth}
      \centering
\includegraphics[trim={90 0 90 10},clip,width=\textwidth]{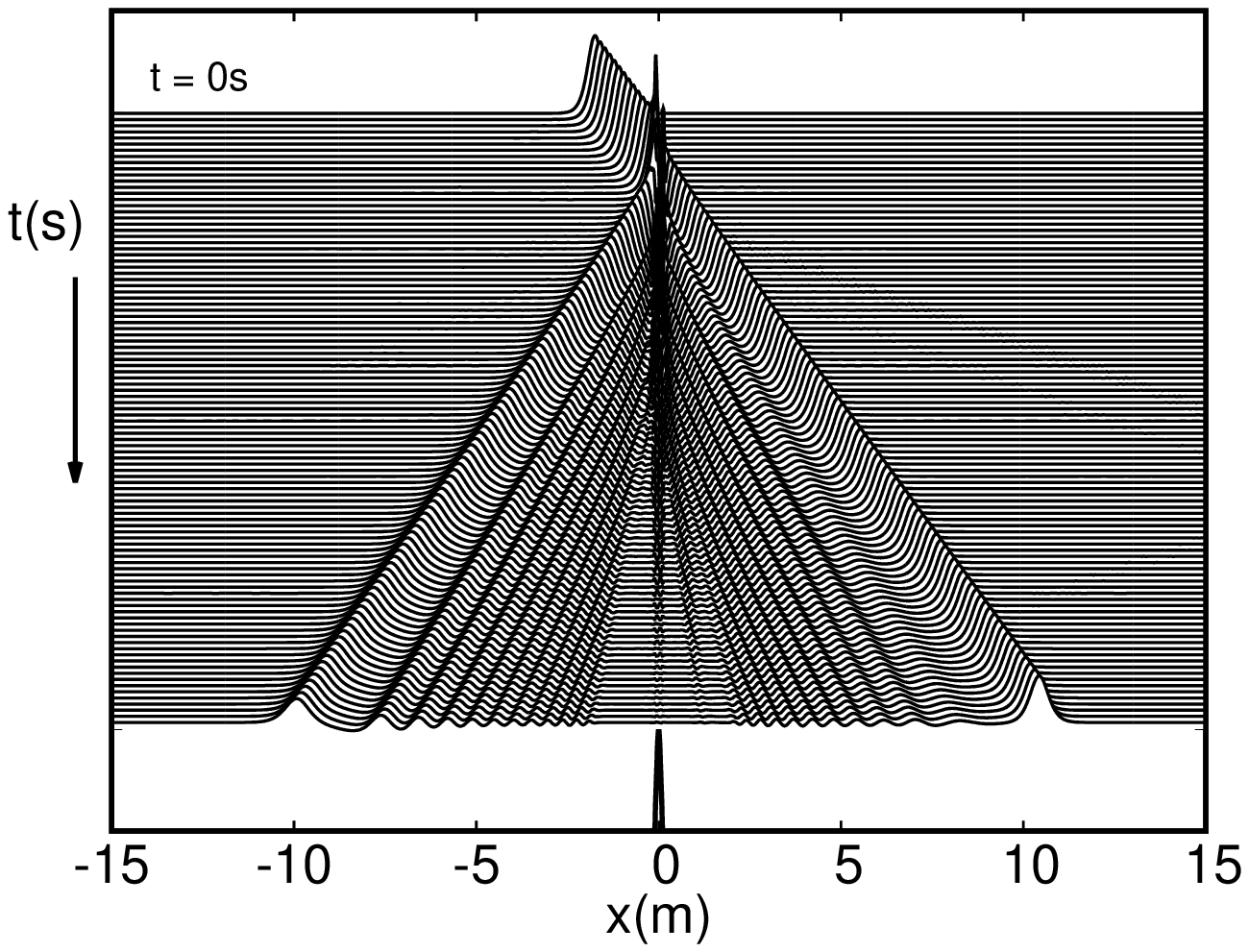}
\caption{$(x,t)$-plane representation of $\waterh+z$ for the incomplete, relaxed \eqref{inc_relaxed}.}\label{fig:triangle_oldModel}
\end{subfigure}
\caption{Solitary wave passing over a triangular obstacle for Exp.~78:
  $\waterh_0=\SI{12.5}{cm},\alpha=\SI{4.75}{cm}$. Final time
  $t=\SI{10}{s}$.} \label{fig:solitary_triangle}
\end{figure}
%%%%%%%%%%%%%
\subsection{1D Solitary wave propagating over a step} \label{Sec:1D_solitary_over_step}
We now consider the problem of a solitary wave propagating over a step
as described in \cite{seabra-santos_renouard_temperville_1987}. Again
here we compare the numerical results to the experimental data. In
particular, we focus on the fission phenomena of the solitary wave due
to the passing over the step. Recall that the fission phenomena is
defined to be the process by which an incident solitary wave evolves
into at least two (transmitted) solitary waves ranked in order of
decreasing amplitude and followed by a small dispersive tail.

 The computational domain is $\Dom=(\SI{-10}{m},\SI{30}{m})$.
The bathymetry for this problem is a step of height
$\SI{0.1}{m}$ defined over $x\geq\SI{0}{m}$ and $\SI{0}{m}$ elsewhere.
Similarly as in the previous section, we introduce a smoothed bathymetry profile
by setting  $z(x) = 0.1(\frac12 + \frac{1}{\pi} \arctan(\frac{x}{d\sqrt{\waterh_0 h}}))$
where $h$ is the mesh-size and $\waterh_0$ is the water depth. The constant $d$ is chosen as in
the previous section. Since the computational domain is finite, we limit the reflection
of waves at the left end of the domain by introducing an ``absorption
zone'' as described in~\cite{Tovar_PhD}. Here, the absorption zone is set to be $D_{\text{abs}} = (\SI{-10}{m},\SI{-5}{m})$.
% table goes here
\begin{table}[h!]\small
\centering
\begin{tabular}{||c c c c c c c c c||}
 \hline
 Exp. & $\waterh_0$ & $\alpha$ & $\alpha_{t1}$ & $E_1$ & $\alpha_{t2}$ & $E_2$ & $\alpha_{t3}$ & $E_3$  \\ [0.5ex]\hline\hline
 1 & 30.0 & 4.25 & 5.58 & 13\% & 1.18 & 44\% & -- & -- \\
 2 & 30.0 & 6.80 & 8.92 & 4.9\% & 1.78 & 11\% & -- & -- \\
 3 & 30.0 & 7.10 & 9.30 & 5.1\%& 1.85 & 6.3\% & -- & -- \\
 4 & 30.0 & 7.50 & 9.81 & 1\% & 1.94 & 11\% & -- & -- \\
 6 & 30.0 & 9.70 & 12.54 & -2.8\% & 2.43  & 2.1\% & -- & --\\
 7 & 25.0 & 1.78 & 2.40 & 8.1\% & 0.69 & -1.4\% & -- & --  \\
 8 & 25.0 & 2.57 & 3.61 & 15\% & 0.96 & 28\% & -- & -- \\
 9 & 25.0 & 3.84 & 5.42 & 14\% & 1.39 & 20\% & -- & -- \\
 10 & 25.0 & 5.75 & 7.96 & -0.25\% & 2.03 & 0\% & -- & --  \\
 11 & 25.0 & 7.17 & 9.79 & -5.3\% & 2.49 & 0.81\% & -- & --  \\
 20 & 20.0 & 1.63 & 2.57 & 8.0\% & 0.90 & -8.2\% & 0.19 & -57\% \\
 21 & 20.0 & 2.08 & 3.26  & 5.8\% & 1.17 & 7.3\%  & 0.21 & -25\%  \\
 22 & 20.0 & 2.43 & 3.77 & 5.9\% & 1.36 & 12\% & 0.24 & 20\%  \\
 23 & 20.0 & 2.93 & 4.49 & 4.4\% & 1.64 & 5.1\% &  0.27 & -21\%  \\
 24 & 20.0 & 3.65 & 5.48 & 3.2\% & 2.01 & 12\% & 0.32 & -3\% \\
 [1ex]
 \hline
\end{tabular}
\caption{Numerical results for the passing of a solitary wave over a shelf (transmitted waves). $\waterh_0$ is the still water depth; $\alpha$ is the amplitude of the incident wave; $\alpha_{tj}$ is the amplitude of the $j$th transmitted wave. All values are in \SI{}{cm}.}
\label{table:shelf_trans}
\end{table}

In the original experiments, seven wave gauges were placed along the
basin to measure the wave heights, and in particular, the height of
the transmitted waves. Recall that transmitted waves are waves that
form after the solitary wave passes over the step and separates into
multiple solitary waves (this can be seen in Figure
\ref{fig:solitary_step}). For our computations, we measure the
amplitude of the transmitted waves at $x = \SI{15}{m}$ since it was
stated in \citep{seabra-santos_renouard_temperville_1987} that ``a
length of about 15 m was required for such a wave-sorting process''.
(We note here that it's not clear at which gauge the transmitted waves
reported in \citep[Tab.~1]{seabra-santos_renouard_temperville_1987} were measured).
We reproduce 15 of the experiments shown in
  \citep[Tab.~1]{seabra-santos_renouard_temperville_1987} and list the
  different values of $\waterh_0$ and $\alpha$ used in Table
  \ref{table:shelf_trans}. The experiment numbers listed in
  Table \ref{table:shelf_trans} coincide with those of
  \citep[Tab.~1]{seabra-santos_renouard_temperville_1987}.  We
  run the computations with a mesh composed of 3200 $\polP_1$ elements with
  $\text{CFL}=0.1$ until the final time $t=\SI{20}{s}$ for Experiments
  1-10 and $t=\SI{25}{s}$ for Experiment 20-24.  In Table
\ref{table:shelf_trans}, we report the results of our computations and
compare the transmitted wave heights to those reported in
\citep{seabra-santos_renouard_temperville_1987}. In the table, $E_j$
is the relative difference between the computational and experiment
values and is defined as
$E_j=\frac{\alpha_j-\alpha_\text{Exp}}{\alpha_\text{Exp}}$ for each
$j$-th transmitted wave. Good agreement with the experimental data is
observed overall. In Figure \ref{fig:solitary_step}, we show the
  graph of the free surface $\waterh + z$ in the $(x,t)$-plane for
  Experiment 10 and 24 listed in Table~\ref{table:shelf_trans}.
% Solitary wave step figure %
\begin{figure}
\centering
\begin{subfigure}{0.49\textwidth}
  \centering
    \includegraphics[trim={100 0 105 10},clip,width=\textwidth]{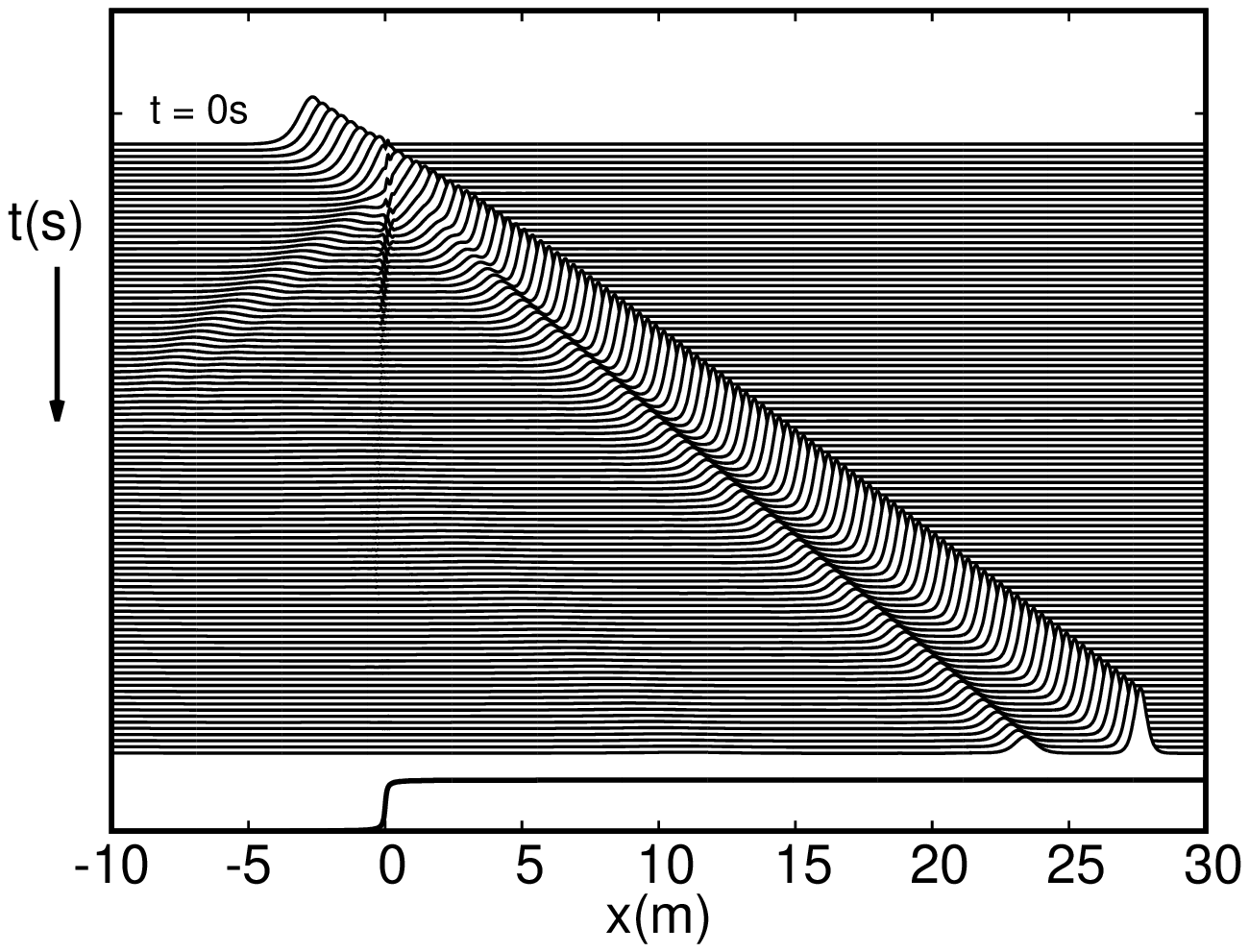}
    % \caption{Relaxed system \eqref{full_relaxed} with full topography effects.}\label{fig:step_newModel}%
\end{subfigure}
\begin{subfigure}{0.49\textwidth}
      \centering
\includegraphics[trim={100 0 105 10},clip,width=\textwidth]{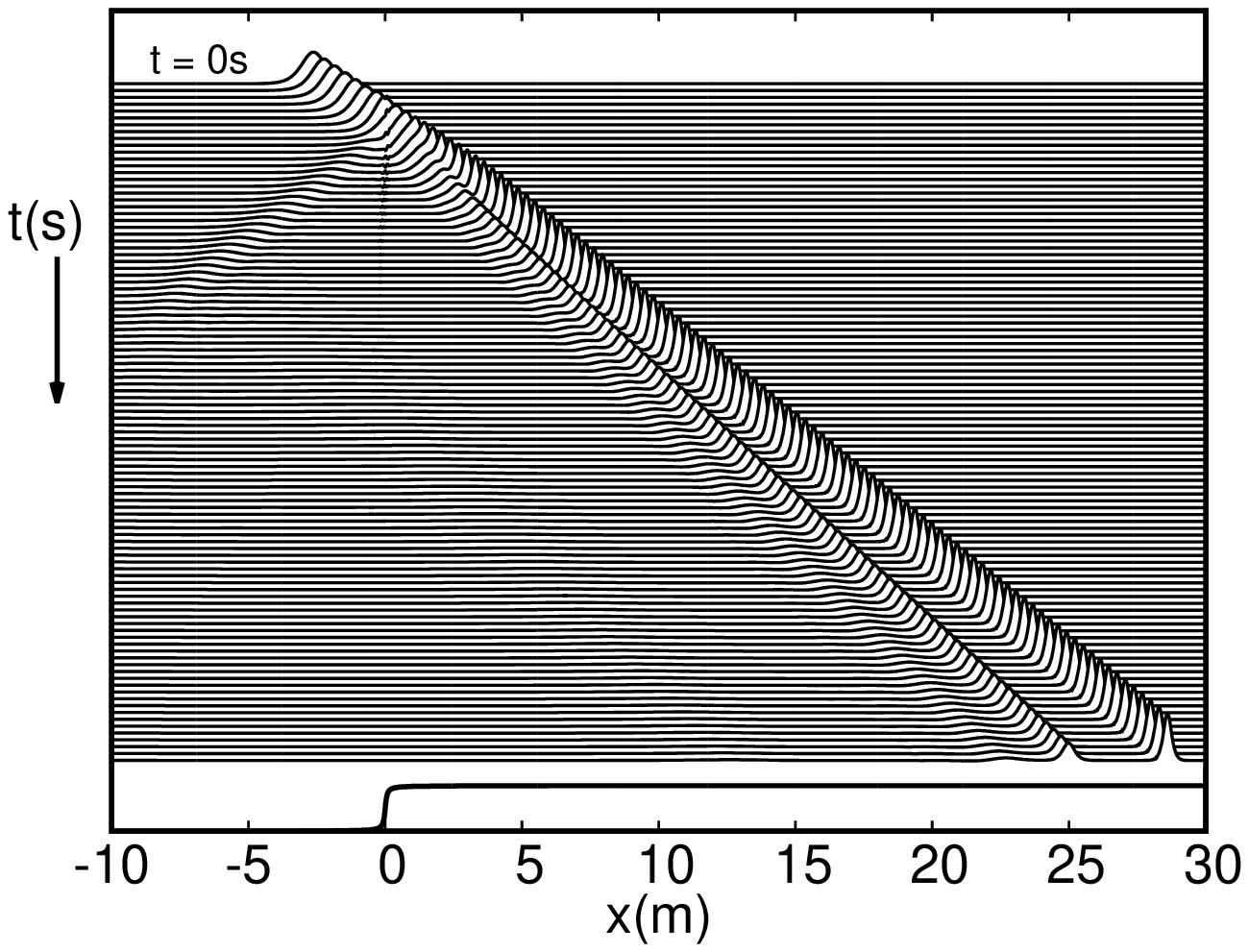}
% \caption{Incomplete, relaxed system\eqref{inc_relaxed}.}\label{fig:step_oldModel}
\end{subfigure}
\caption{Time evolution plot of a solitary wave propagating over a discontinuous step
for Exp.~10 and Exp.~24. Final time $t=\SI{20}{s}$ and $t=\SI{25}{s}$, respectively.} \label{fig:solitary_step}
\end{figure}
%%%%%%%
\bal{
\subsection{1D Shoaling of solitary waves over sloped beach}\label{Sec:solitary_shoaling}
We now consider the 1994 experiments of~\cite{guibo1994} conducted at
LEGI (Laboratoire des \'Ecoulements G\'eophysiques et Industriels) in
Grenoble, France, to investigate the shoaling of solitary waves over a
sloped beach.

We consider 4 series of experiments proposed in~\citep{guibo1994} with a
reference water depth of $\waterh_0 = \SI{0.25}{m}$ and different solitary
wave amplitudes (see: Table~\ref{Tab:shoaling_table}). We simulate
the experiments in one spatial dimension and reproduce the bathymetry as follows:
\[
  z(x) =
  \begin{cases}
    \frac{1}{30}(x-2.5) - \waterh_0, & x \geq 25\\
    -\waterh_0, & \text{otherwise}.
  \end{cases}
\]
\begin{table}[h!]
\centering
\begin{tabular}{||c |c| c| c| c||}
 \hline
 $\,$ & Case~1 &  Case~2 &  Case~3 &  Case~4 \\ [0.5ex]\hline
 $\alpha/\waterh_0$ & 0.096 & 0.2975 & 0.456 & 0.5343 \\ \hline
 WG1 & $\SI{7.75}{m}$ & $\SI{5.75}{m}$ & $\SI{4.25}{m}$ & $\SI{4.25}{m}$\\ \hline
 WG2 & $\SI{8.25}{m}$ & $\SI{6.25}{m}$ & $\SI{5.0}{m}$ & $\SI{5.0}{m}$\\ \hline
 WG3 & $\SI{8.75}{m}$ & $\SI{6.75}{m}$ & $\SI{5.75}{m}$ & $\SI{5.75}{m}$\\ \hline
\end{tabular}
\caption{Solitary wave shoaling experiment~\citep{guibo1994} -- configuaration values}%
\label{Tab:shoaling_table}%
\end{table}
The computational domain is set to $D = (\SI{-5}{m},\SI{35}{m})$. For
each experiment, we initialize the solitary wave at $x_0 = \SI{0}{m}$
with the profiles defined in~\eqref{solitary_wave} and the amplitudes
shown in Table~\ref{Tab:shoaling_table}. We run the computations to
the final time $T=\SI{10}{s}$ on the three difference meshes with
respective mesh-size:
$h = \{\SI{0.05}{m}, \SI{0.025}{m}, \SI{0.0125}{m}\}$ (corresponding
to $800, 1600, 3200$ $\polP_1$ elements).  The CFL number is set to
$0.1$. We set wall boundary conditions at both ends of the domain.}

\bal{In the experiments, the wave elevation was measured with three
  wave gauges (WGs) which were moved for each case. We report the
  location of the wave gauges in Table~\ref{Tab:shoaling_table}.  In
  Figure~\ref{Fig:shoaling}, we show the comparisons with the
  numerical computations and the experimental data for each case. We
  observe that the numerical results match the experimental data
  reasonably well. This set of experiments reinforces the observations
  made in \S\ref{Sec:1D_solitary_over_triangle} and
  \S\ref{Sec:1D_solitary_over_step} that the dispersive Serre model
  with topography effects captures well the shoaling phenomenon
  induced by topography.}
\begin{figure}[h]
\centering
    \includegraphics[trim={0 2 17 17},clip,width=0.47\linewidth]{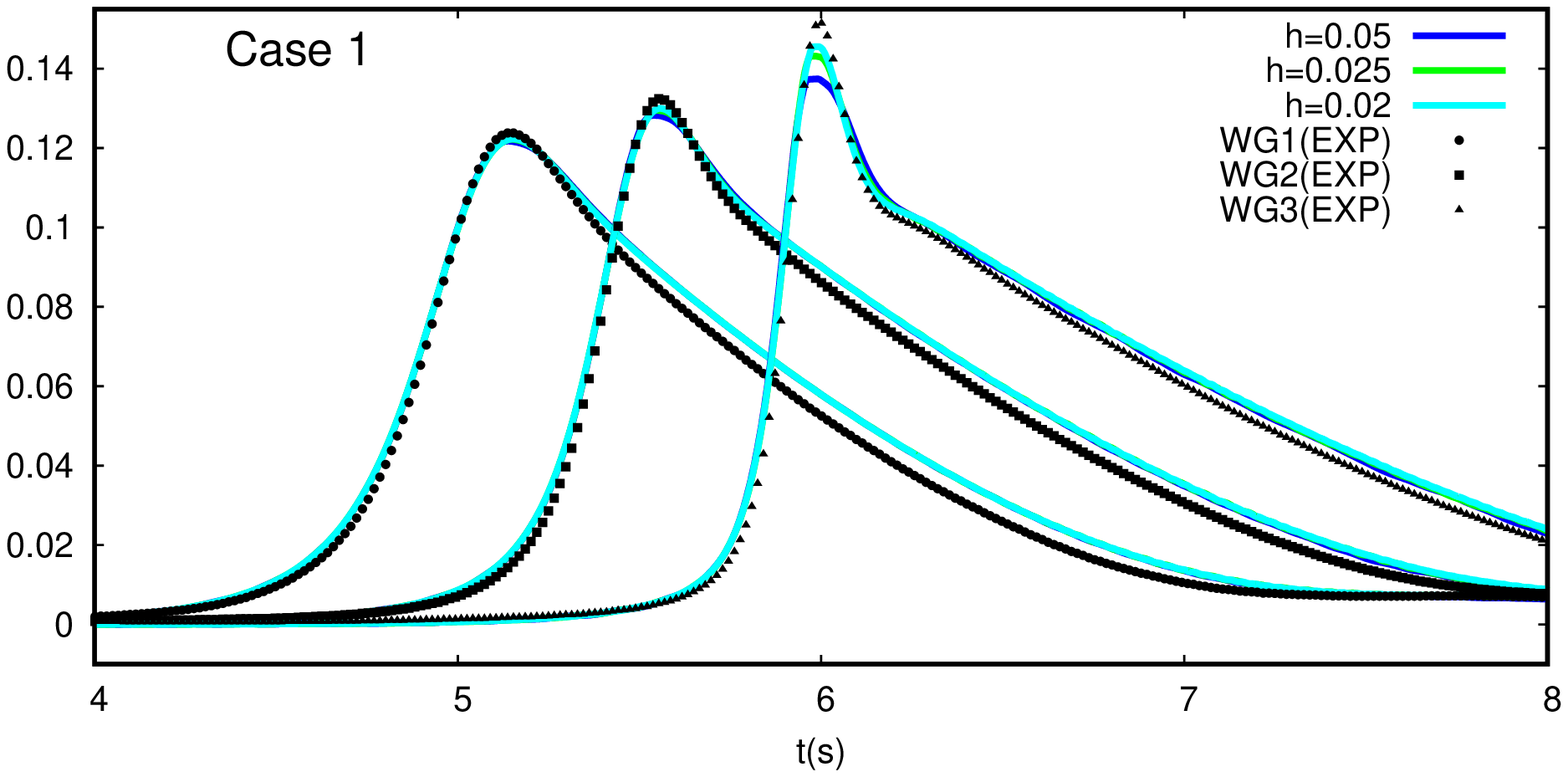}
    \includegraphics[trim={0 2 17 17},clip,width=0.47\linewidth]{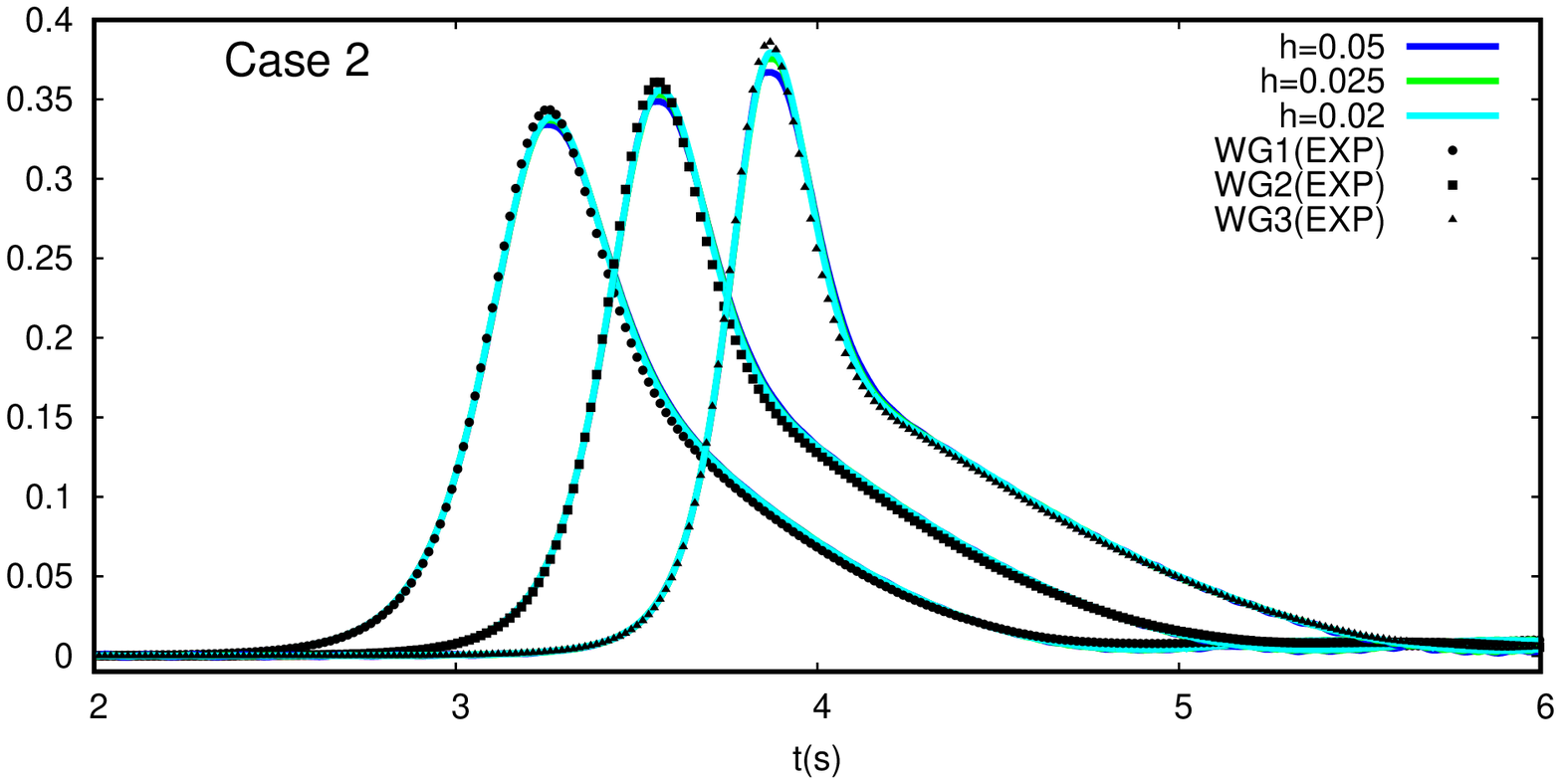}\par
    \includegraphics[trim={0 2 17 17},clip,width=0.47\linewidth]{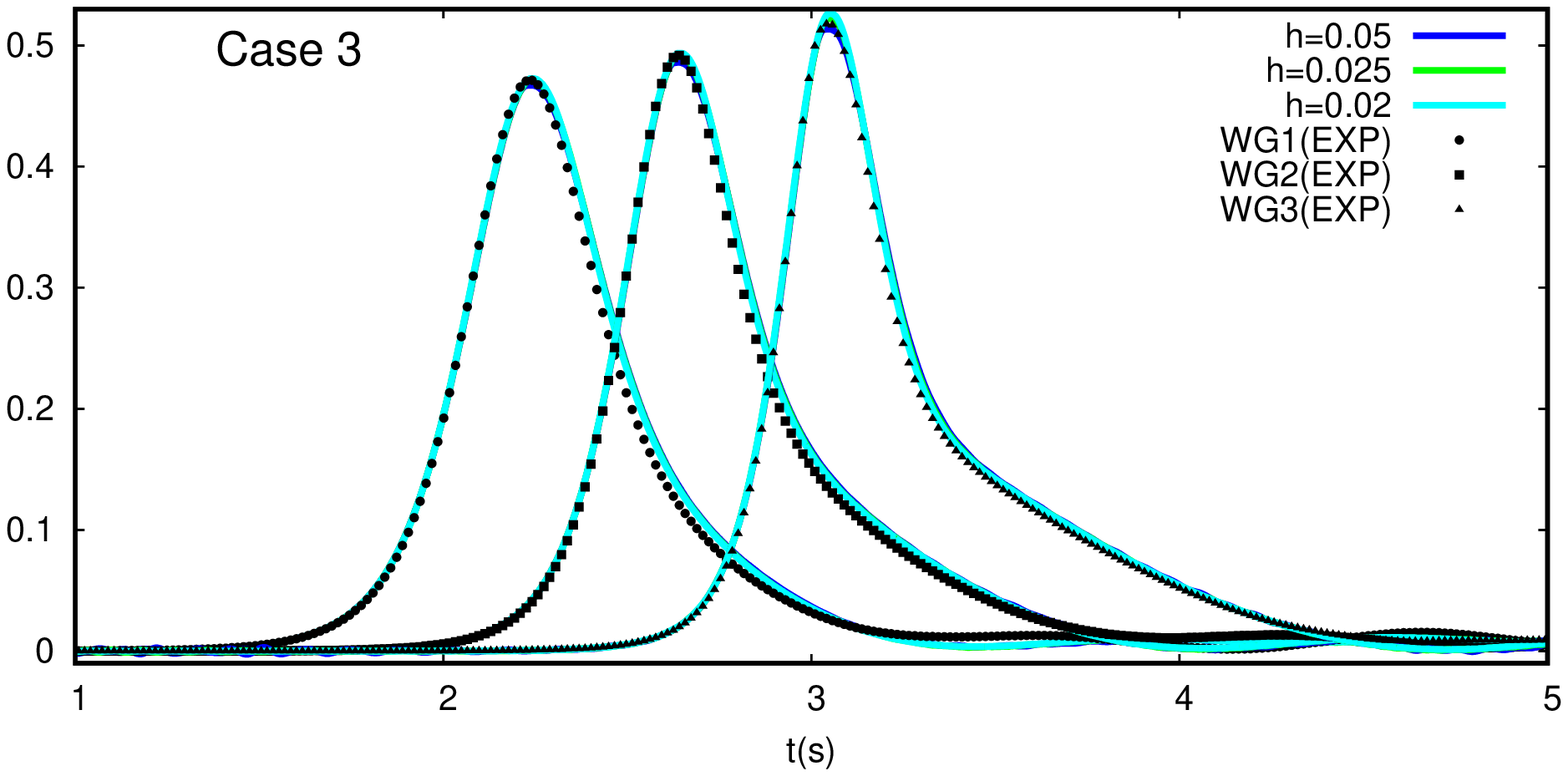}
    \includegraphics[trim={0 2 17 17},clip,width=0.47\linewidth]{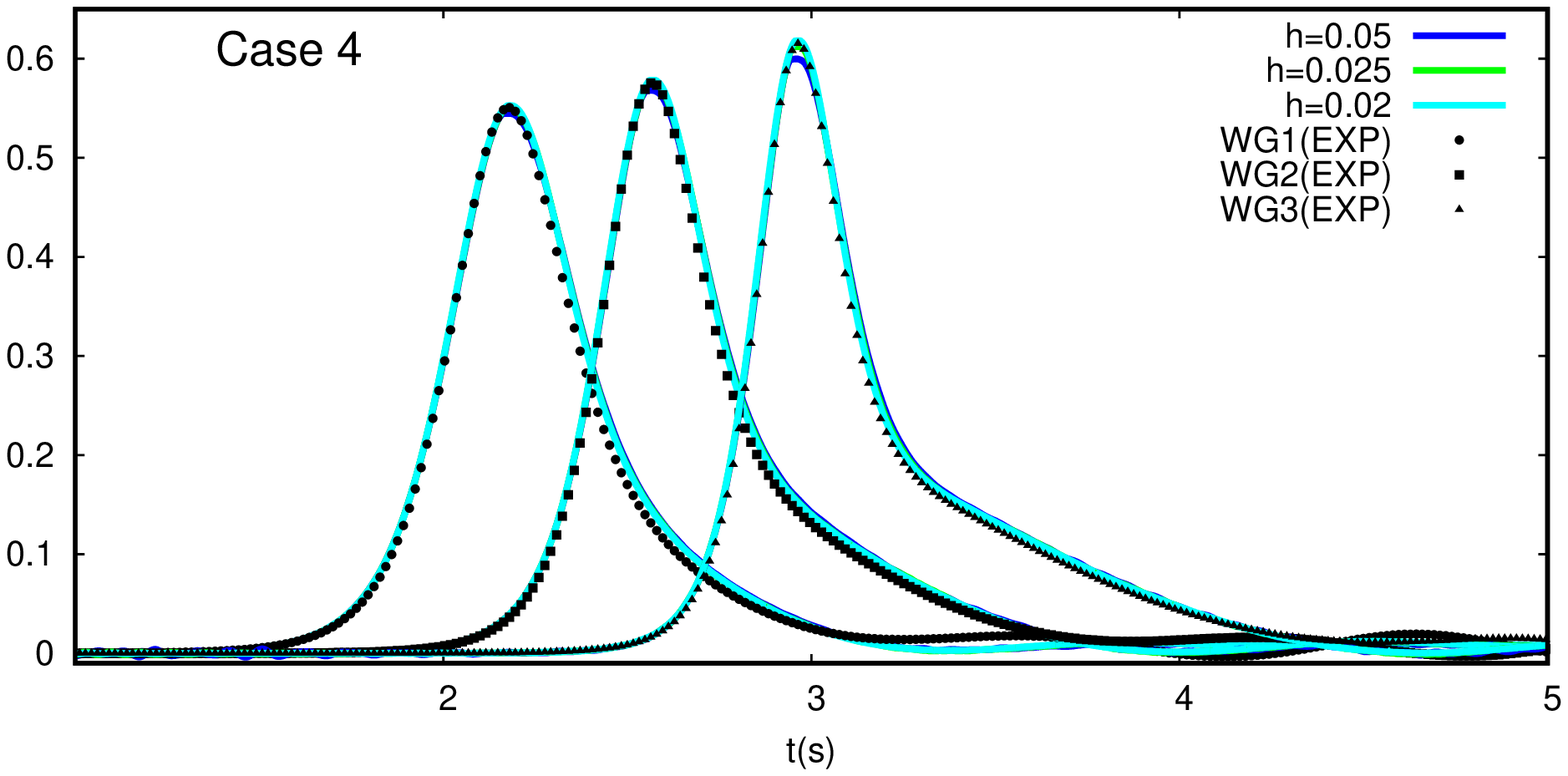}
    \caption{Comparison of numerical results with experimental data for solitary wave shoaling
    experiments of~\citep{guibo1994}.}
    \label{Fig:shoaling}%
\end{figure}
%
%%%%%%%
\subsection{1D Periodic waves propagation over a submerged bar}
\label{Sec:beji}
We now consider the 1994 experiments conducted in \cite{BEJI_1994}
which investigate the propagation of periodic waves over a submerged
trapezoidal bar.  The goal of the experiments is to model the
interaction of highly dispersive waves, and in particular, the release
of higher-harmonics into a deeper region after the shoaling process.

We consider two of the experimental setups described in
\citep{BEJI_1994}: \textup{(i)} sinusoidal long waves (SL) with target
amplitude $a = \SI{1}{cm}$ and period $T_p = \SI{2}{s}$;
\textup{(ii)} sinusoidal high-frequency waves (SH) with target
amplitude $a = \SI{1}{cm}$ and period $T_p = \SI{1.25}{s}$.  We
simulate these experiments in one spatial dimension and reproduce the
bathymetry of the submerged bar as follows:
\[
  z(x) =
  \begin{cases}
    \frac{1}{20}(x-6), & 6 \leq x \leq 12\\
    0.3, & 12 \leq x \leq 14\\
    0.3 - \frac{1}{10}(x-14), & 14 \leq x \leq 17\\
    0, & \text{otherwise}.
  \end{cases}
\]
\begin{figure}[h]
\centering
    \includegraphics[trim={20 20 10 20},clip,width=0.65\linewidth]{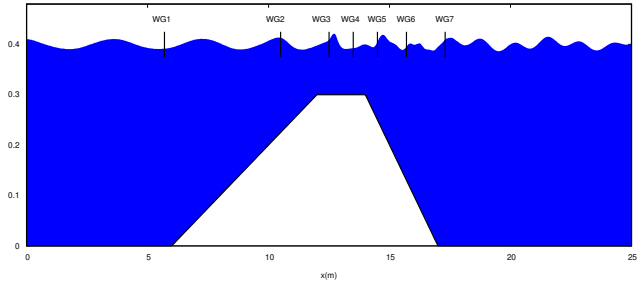}
    \caption{Submerged bar set up with gauge locations.}\label{fig:bar_setup}%
\end{figure}
The computational domain is set to be
$D=(-\SI{12.3}{m},\SI{37.7}{m})$. We impose two relaxation zones in
the domain for the generation and absorption of waves (we refer the
reader to \cite{Tovar_PhD} for the details). The length of the
generation zone for the SL case is $\SI{6}{m}$ (approximately 1.5
wavelengths) and $\SI{4}{m}$ for the SH case (approximately 2.0
wavelengths).  The absorption zone for both cases is set to
$D_{\text{abs}}=(\SI{25}{m},\SI{37.7}{m})$.  We set the reference
water depth to $\sfH_0 = \SI{0.4}{m}$ and initialize the water height
profile with $\waterh_0(x) = \sfH_0 - z(\bx)$ and discharge
$q_0(x)=0$. The periodic waves are introduced into the domain via the
generation zone with the profiles given by:
\begin{equation*}
  \waterh(x, t) = \waterh_0 + a \sin(k x - \sigma t),\quad u(x, t) = \frac{a}{\waterh_0}\frac{\sigma}{k}\sin(k x - \sigma t),
\end{equation*}
where $a$ is the amplitude, $k$ the wave number and $\sigma$ the wave
frequency.  Here, we define the wave frequency by
$\sigma=\tfrac{2\pi}{T_p}$ and $k$ is found by using the dispersion
relation for the full Serre model:
$k^2 = 3 \sigma^2/ (3g\waterh_0 - \waterh_0^2 \sigma^2)$. We set the final time to be
$t=\SI{60}{s}$ and run with CFL=0.175. We run the computations on
three different meshes with mesh-size
$h = \{\SI{0.05}{m},\SI{0.025}{m},\SI{0.0125}{m}\}$ (corresponding
to $1000$, $2000$, and $4000$ $\polP_1$ cells.).

\begin{figure}[h]
\centering
    \includegraphics[trim={0 20 5 5},clip,width=0.35\linewidth]{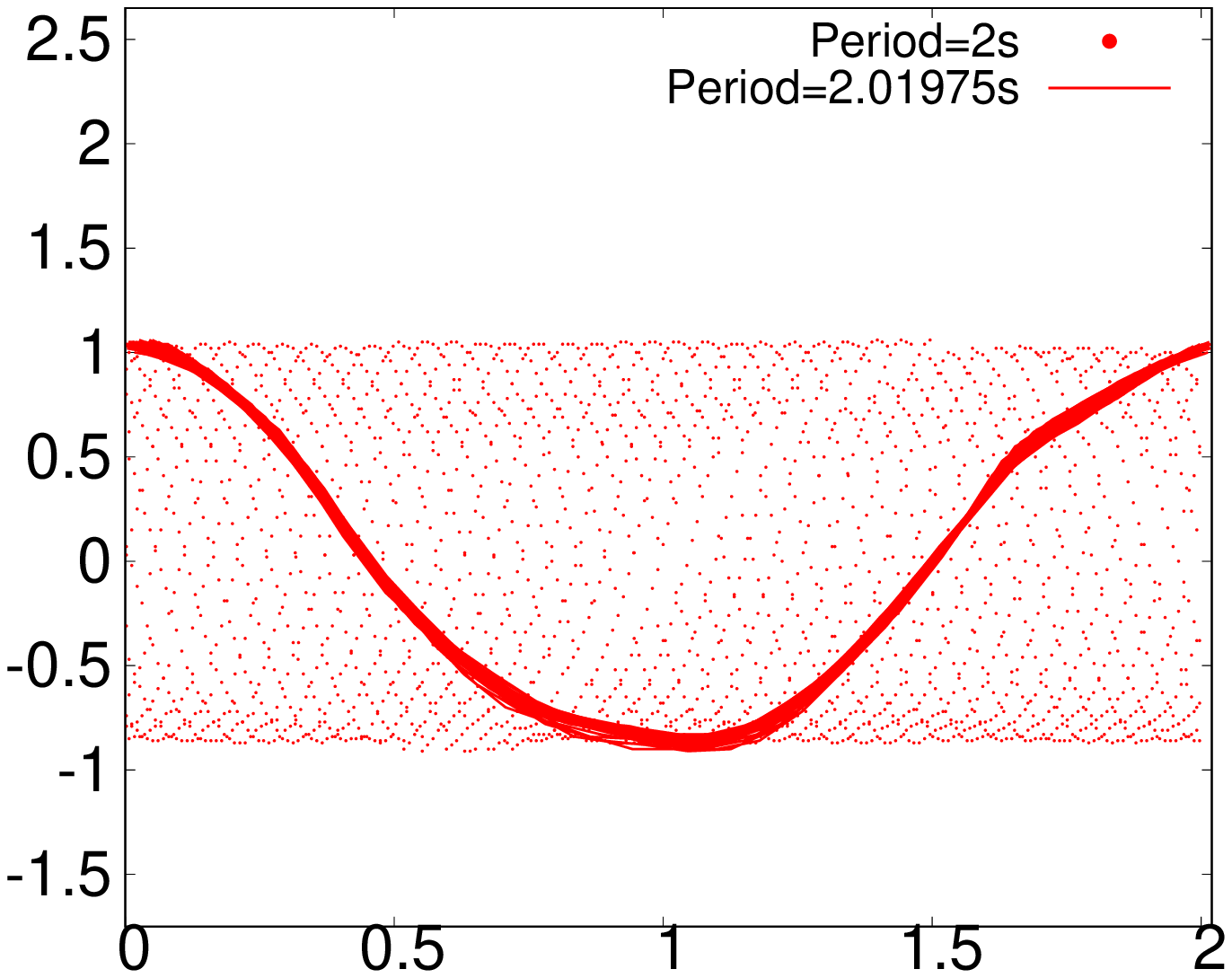}
    \includegraphics[trim={0 20 5 5},clip,width=0.35\linewidth]{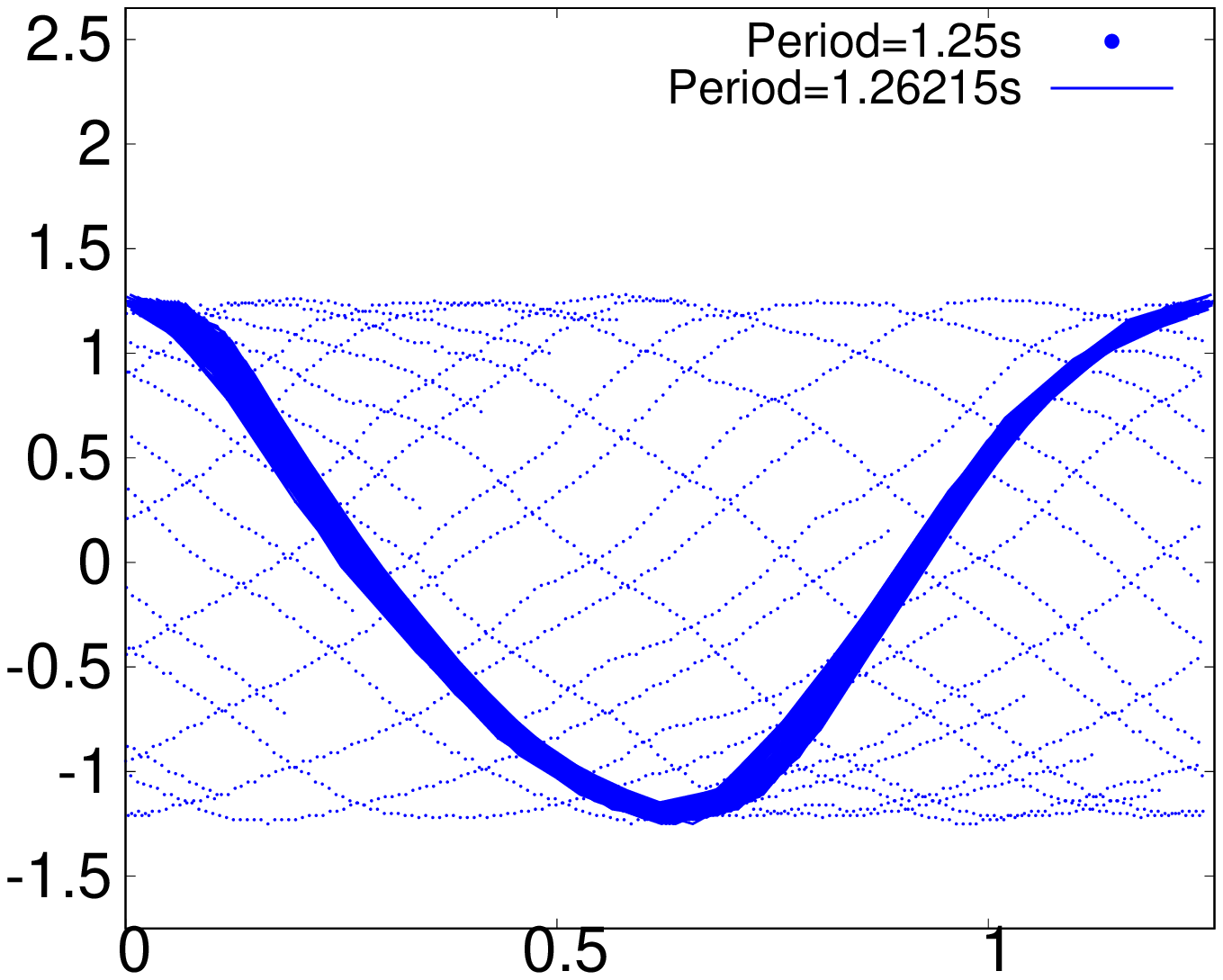}
    \caption{Illustration of \textit{period-folding} with experimental wave gauge 1 with SL case (left) and SH case (right).}
    \label{fig:period_folding}%
  \end{figure}
In the original experiments, seven wave gauges (WGs) were used to
measure the water elevation: WG1($x = \SI{5.7}{m}$),
WG2($x = \SI{10.5}{m}$), WG3($x = \SI{12.5}{m}$),
WG4($x = \SI{13.5}{m}$), WG5($x = \SI{14.5}{m}$),
WG6($x = \SI{15.7}{m}$), WG7($x = \SI{17.3}{m}$). In Figure
\ref{fig:bar_setup}, we show the locations of these wave gauges with
respect to the bathymetry. The experimental data used here was
obtained from the original author of the experiments, Serdar Beji.  It
was our experience that the experimental data did not quite match the
targeted values of the period mentioned above.  To illustrate this, we
introduce a post-processing technique of the experimental data that we
call \textit{period-folding}.  The idea is that given some
experimental time series that is supposedly periodic with period $T_p$
in the time interval $[t_0,T_{\text{final}}]$, the folding of the
sequence obtained by the mapping
$t\mapsto t-t_0 - \lfloor\frac{t-t_0}{T_p}\rfloor T_p$ should
represent the evolution of the signal during one period (here
$\lfloor\SCAL\rfloor$ is the floor function).  Doing this folding
gives a better idea of the long time behavior of the experimental data
than just looking at one specific window of length $T_p$ as often done
in the literature. In particular it reveals whether the signal is
indeed periodic with period $T_p$. In Figure~\ref{fig:period_folding},
we use period folding for the experimental data at WG1 with the
targeted values of $T_p$.  This process shows that the experimental
data have not exactly the alleged period. We have been able to
discover a good approximation of the actual period $T_p^{\text{adj}}$
by doing the period folding with various values of $T_p$: \textup{(i)}
SL case, $T_p=\SI{2}{s}$, $T_p^{\text{adj}} = \SI{2.01975}{s}$;
\textup{(ii)} SH case, $T_p=\SI{1.25}{s}$,
$T_p^{\text{adj}} = \SI{1.26215}{s}$. We also note that the wave
amplitude value for the SH case is closer to $\SI{0.014}{m}$ and this
is what we use for our computations.
  % tabular keeps images neater here
% 1994, sl figure
\begin{figure}[h]
\centering
\renewcommand*{\arraystretch}{0}
\begin{tabular}{*{3}{@{}c}@{}}
  \includegraphics[trim={0 25 10 5},clip,width=0.33\linewidth]{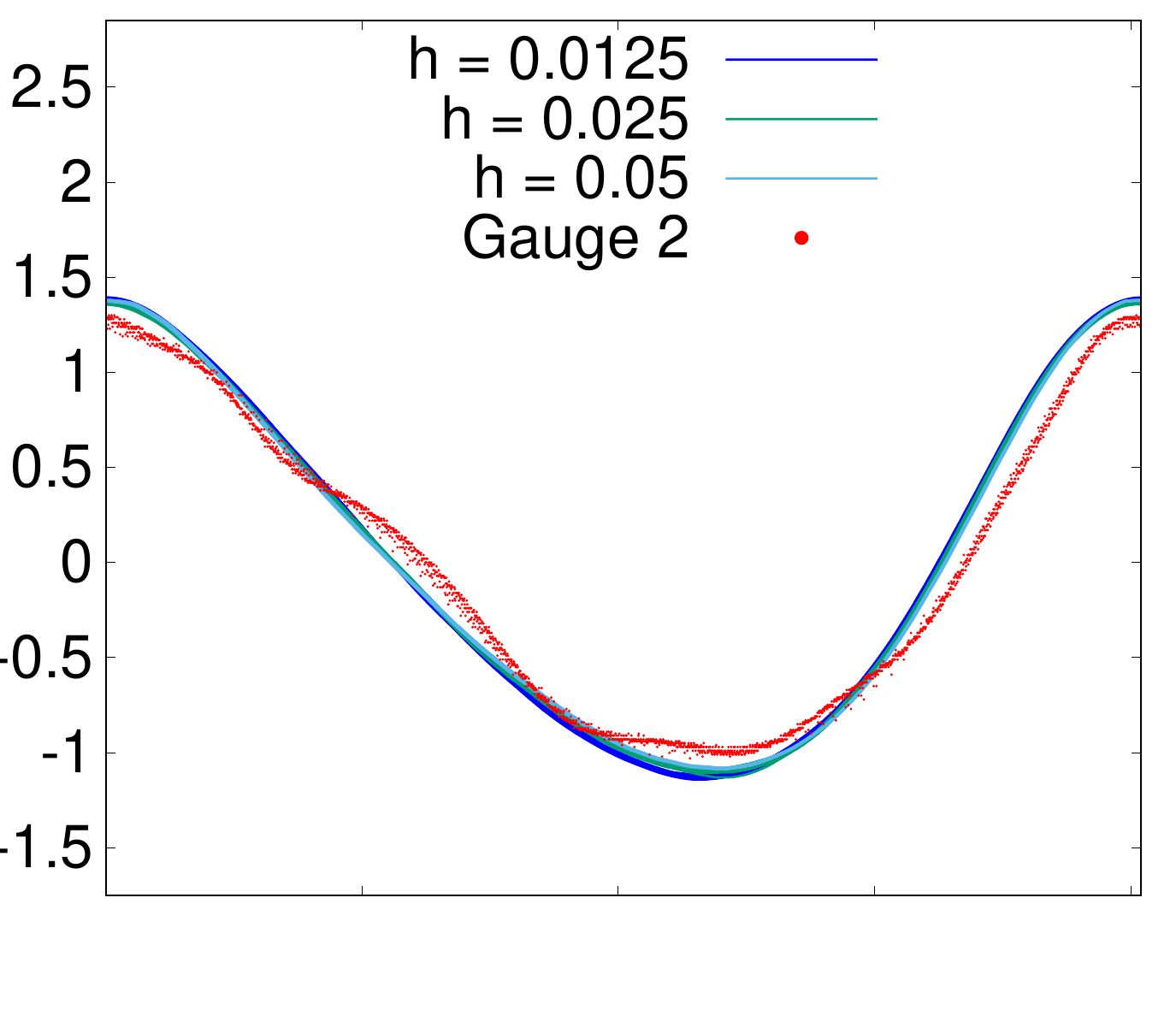} &
  \includegraphics[trim={0 25 10 5},clip,width=0.33\linewidth]{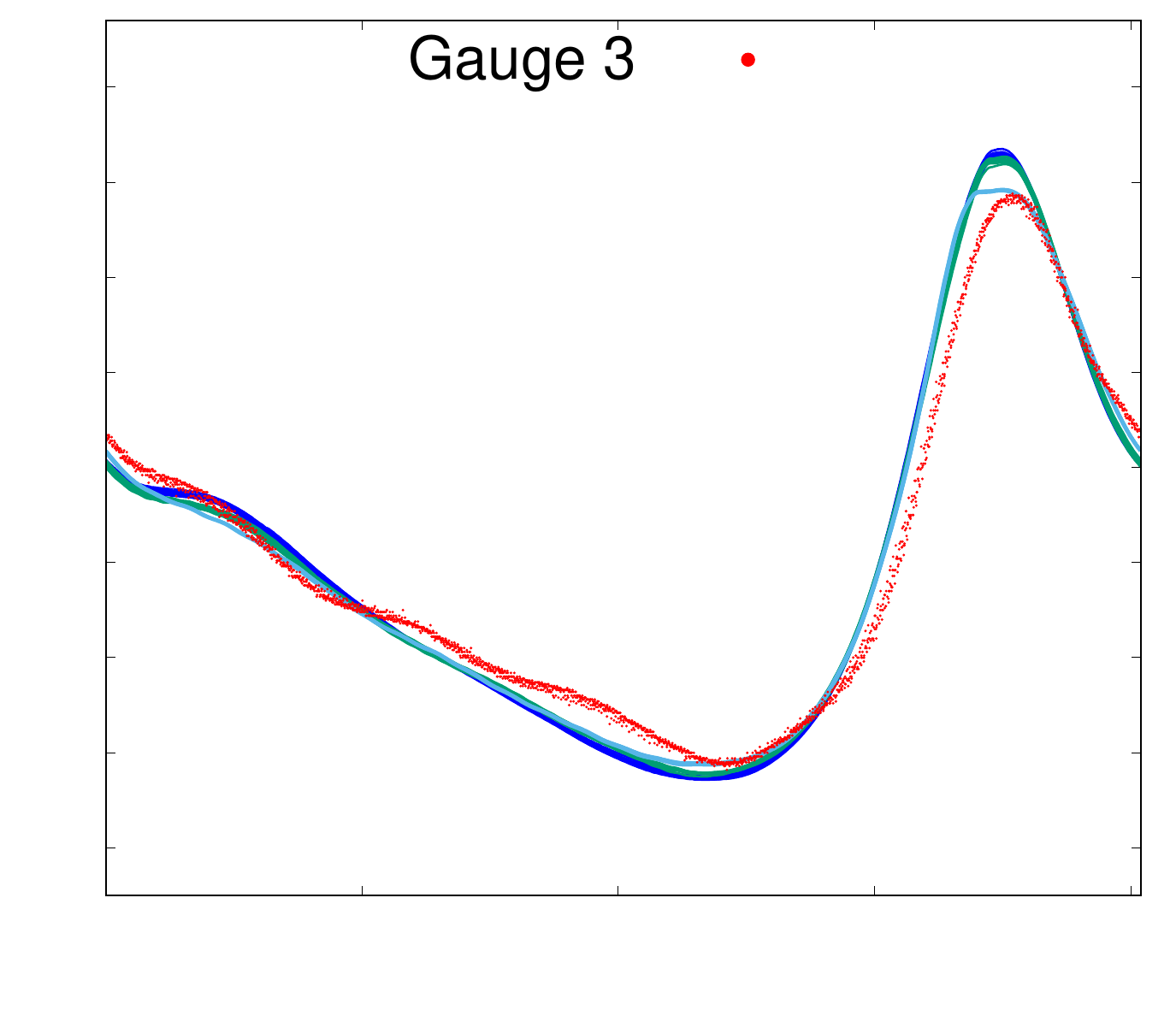} &
  \includegraphics[trim={0 25 10 5},clip,width=0.33\linewidth]{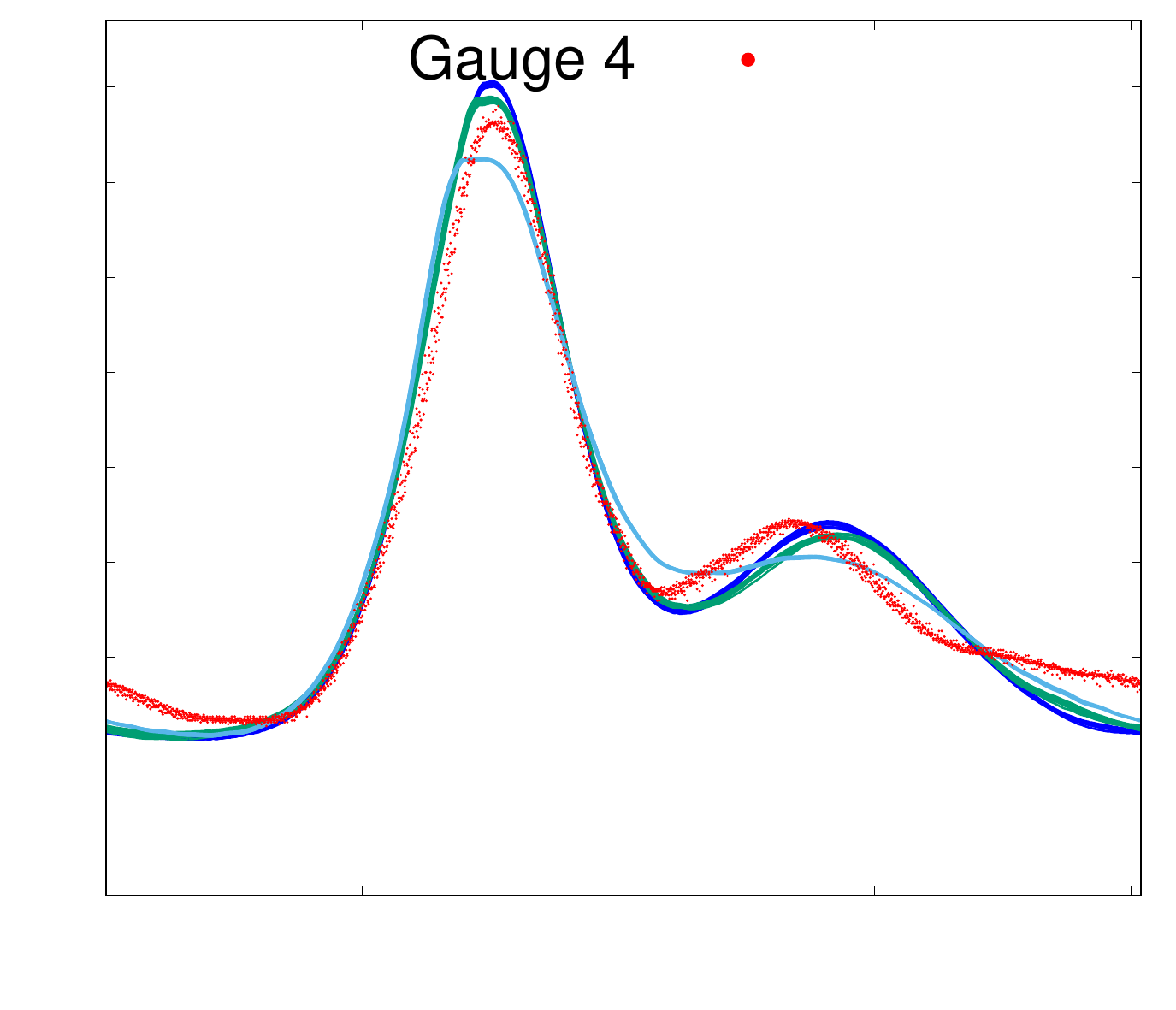} \\
  \includegraphics[trim={5 0 10 5},clip,width=0.33\linewidth]{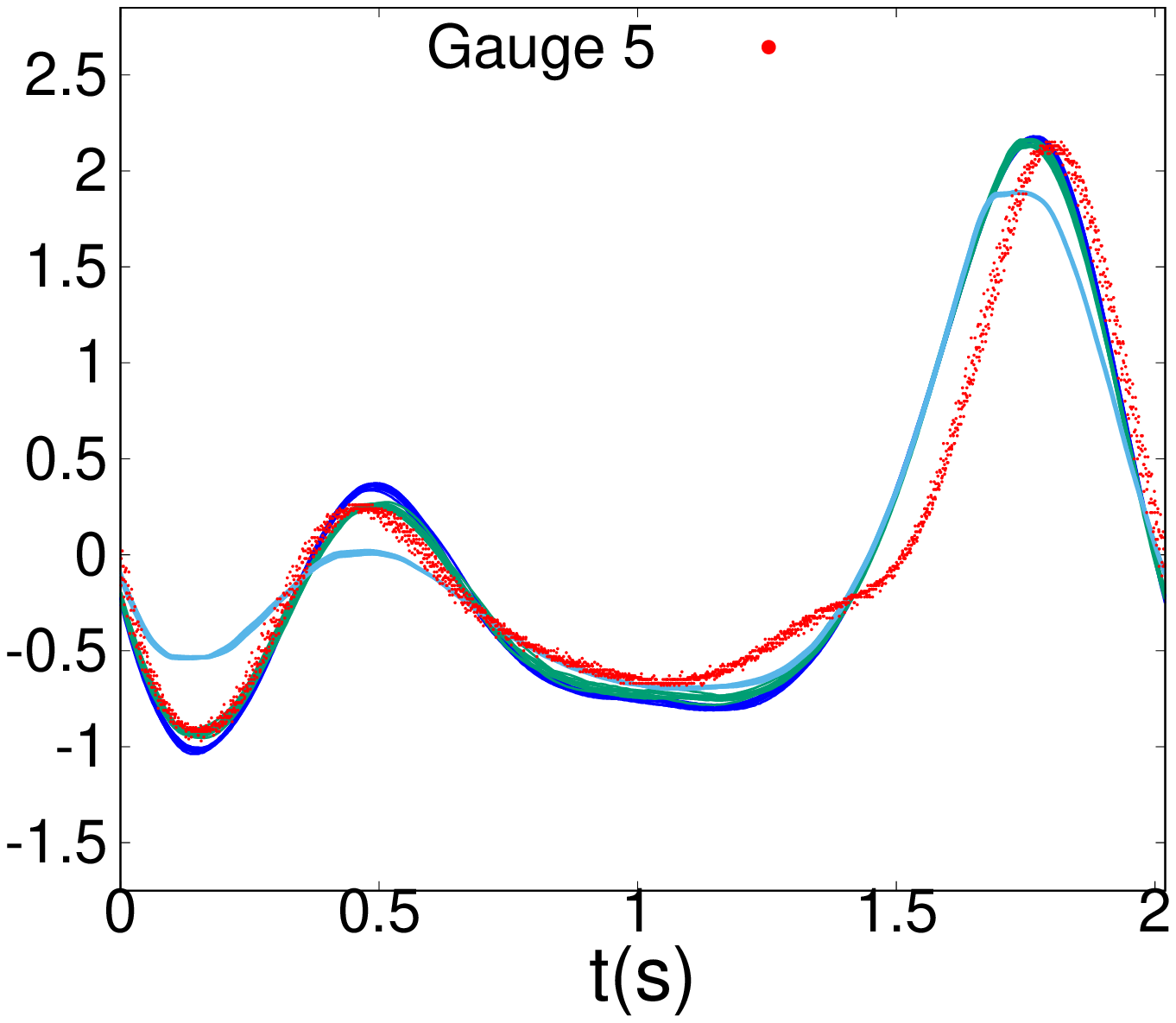} &
  \includegraphics[trim={5 0 10 5},clip,width=0.33\linewidth]{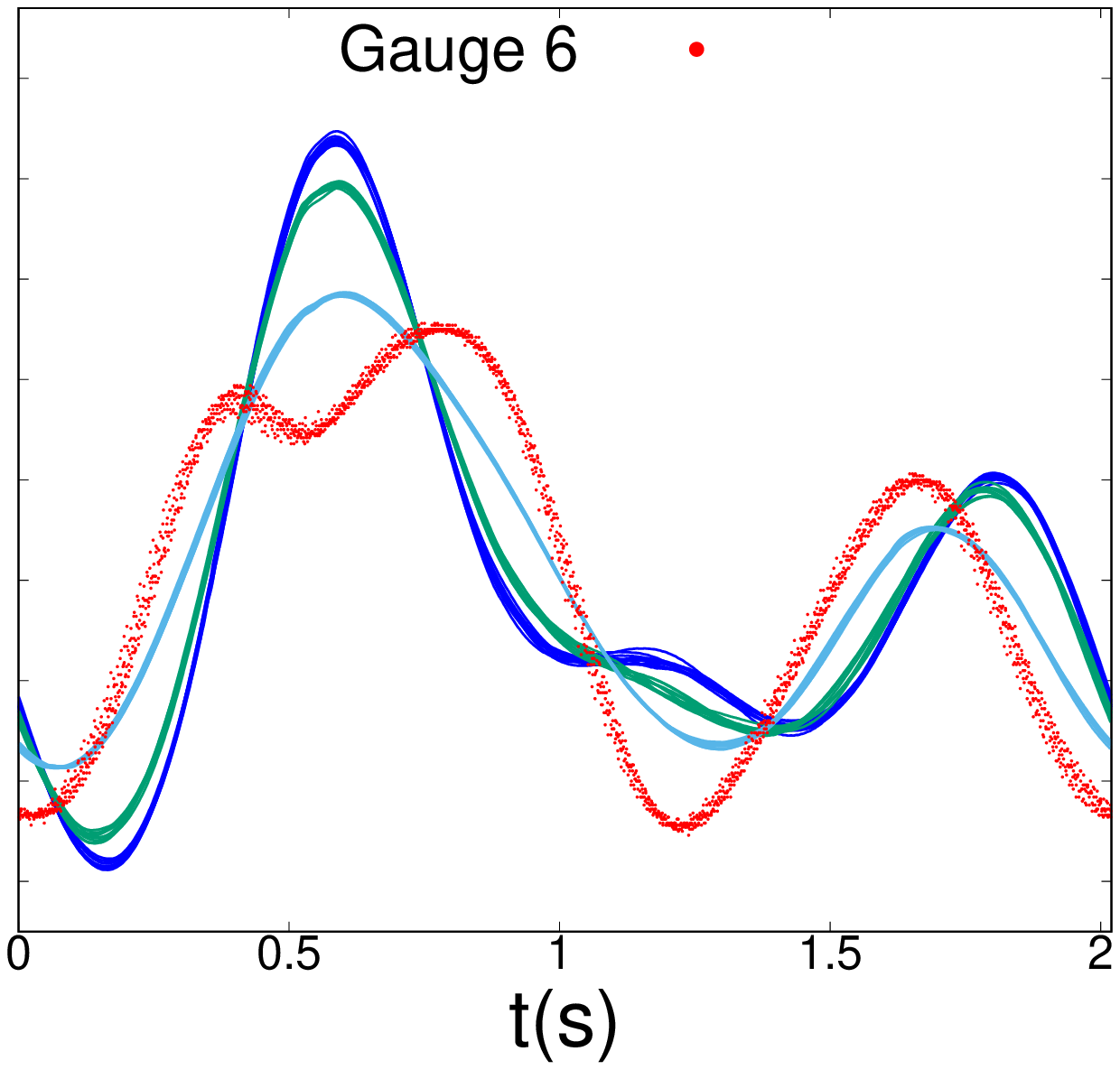} &
  \includegraphics[trim={5 0 10 5},clip,width=0.33\linewidth]{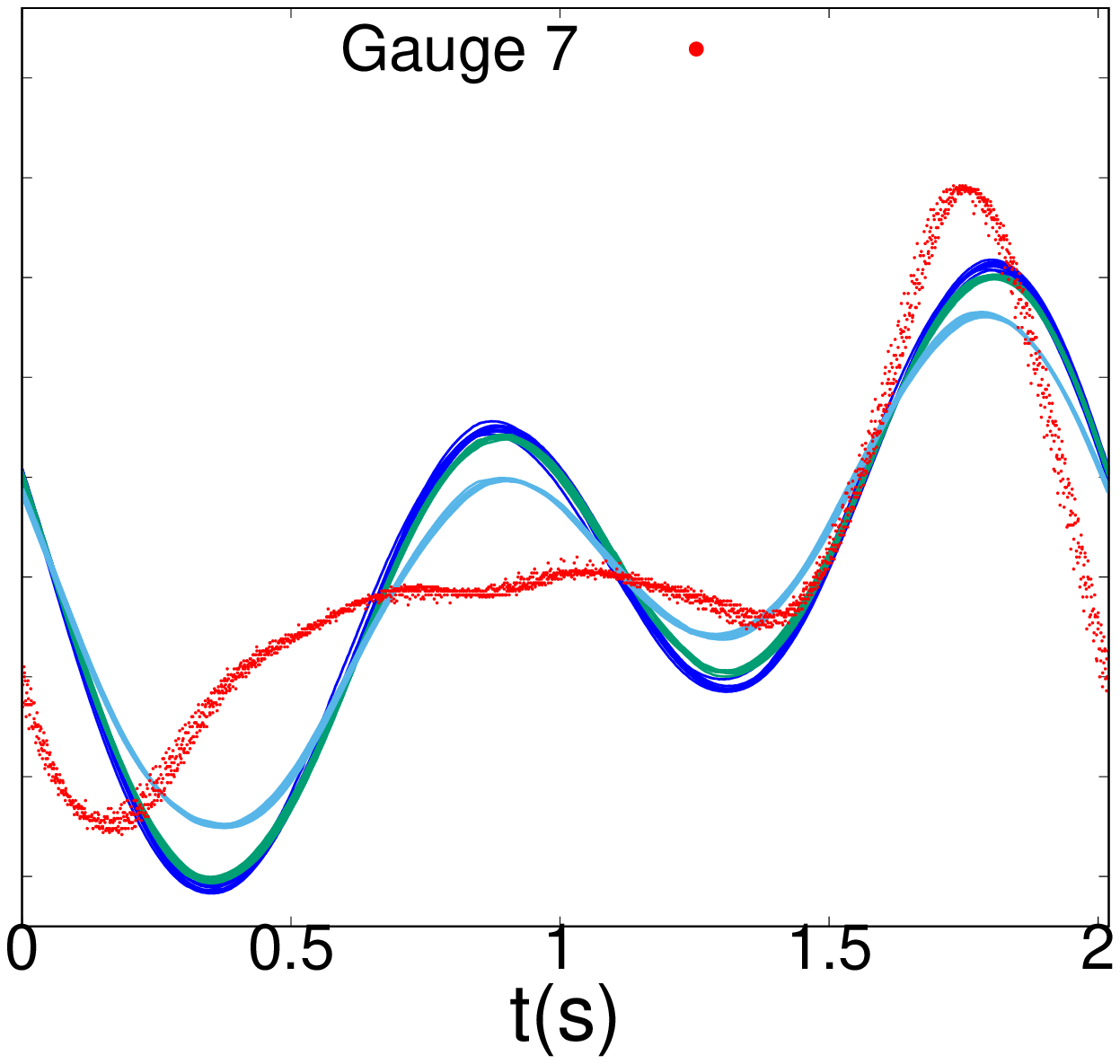}
    \end{tabular}
    \caption{SL Case. Water elevation at seven gauges.
      Numerical results using three meshes, $h = \{\SI{0.05}{m},\SI{0.025}{m},\SI{0.0125}{m}\}$ (solid lines).
      Experimental data (red points).}\label{fig:sl_bar}%
\end{figure}
% 1994, sh figure
\begin{figure}[h]
\centering
\renewcommand*{\arraystretch}{0}
\begin{tabular}{*{3}{@{}c}@{}}
    \includegraphics[trim={0 25 10 0},clip,width=0.33\linewidth]{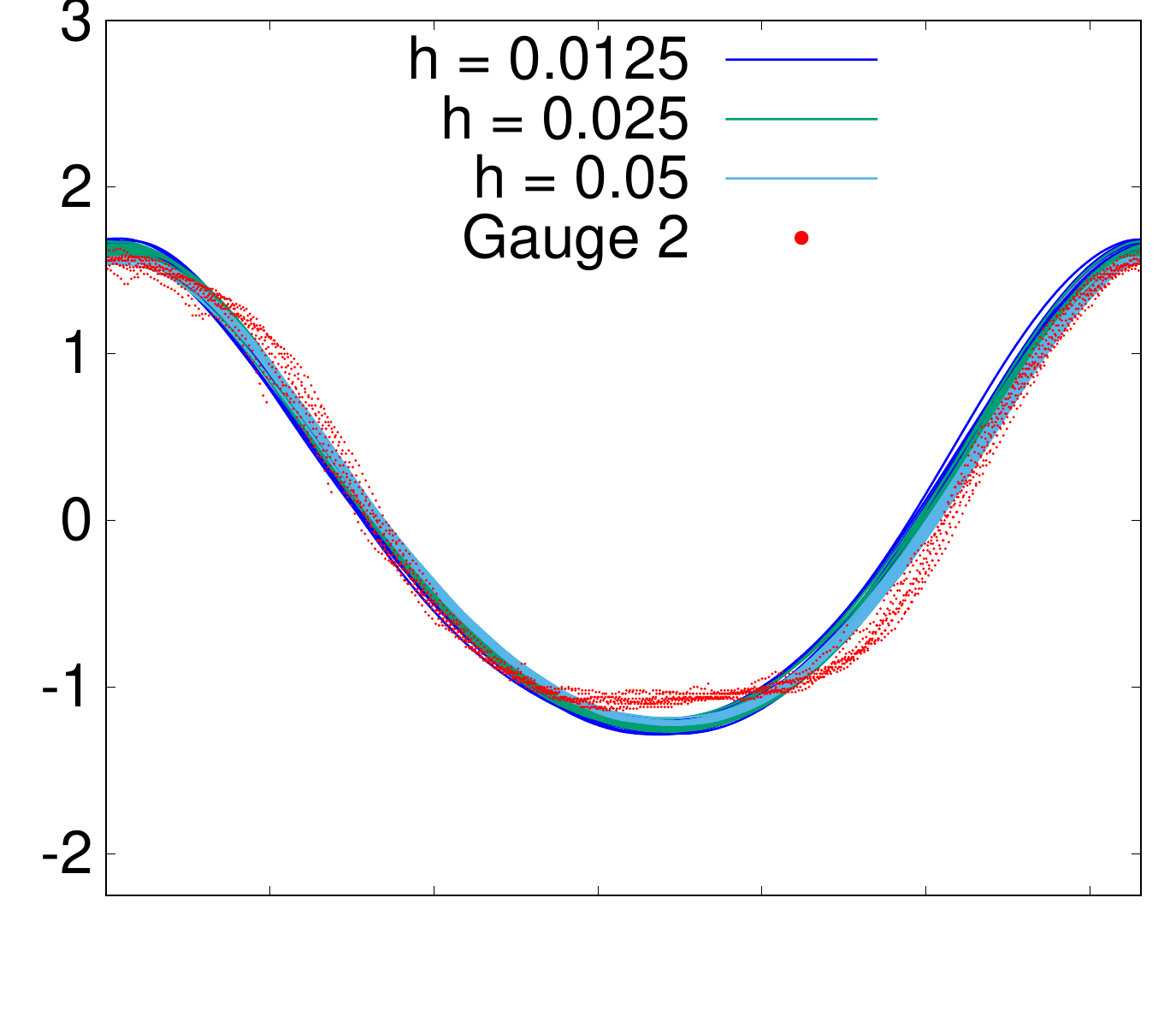} &
    \includegraphics[trim={0 25 10 0},clip,width=0.33\linewidth]{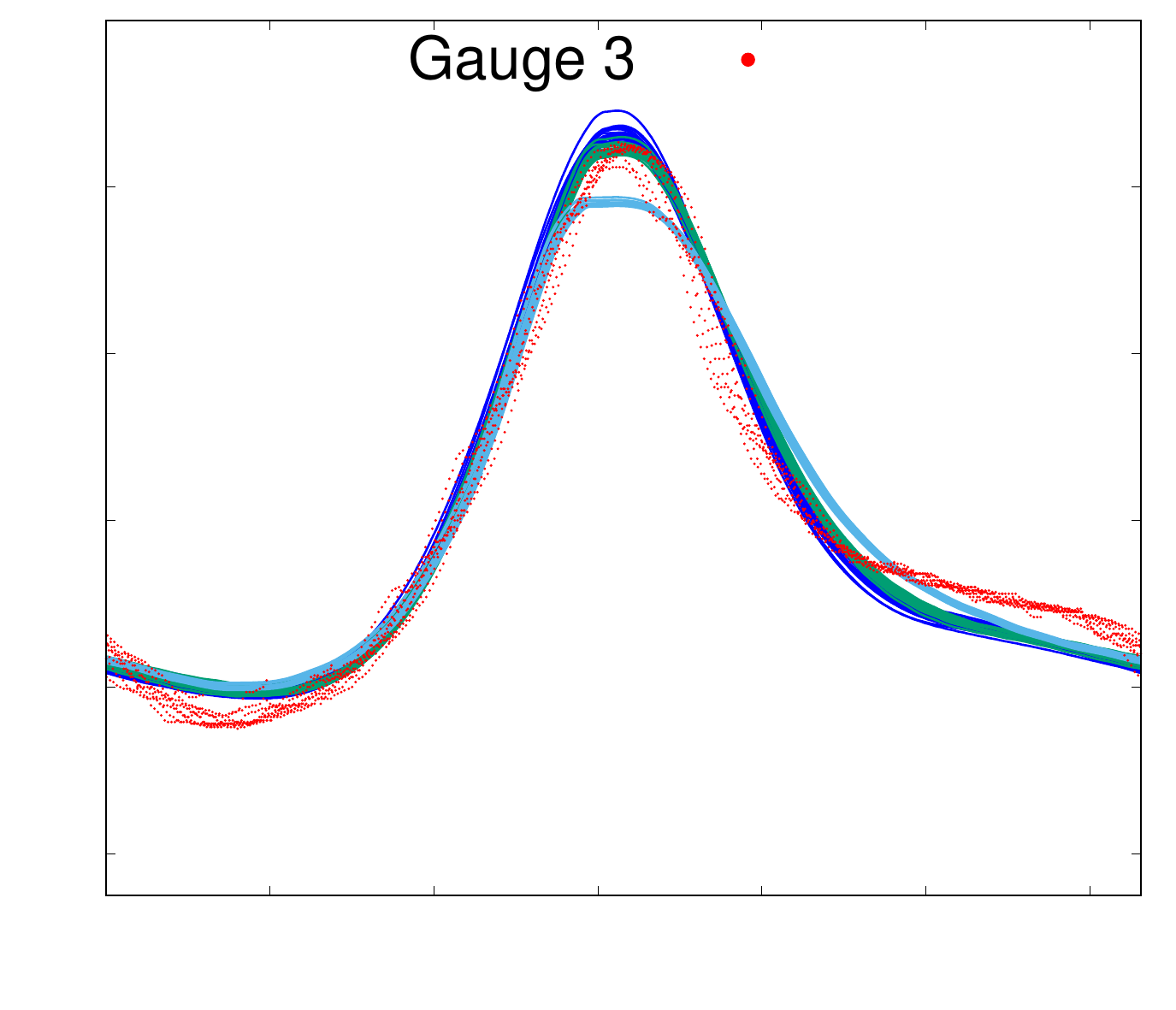} &
    \includegraphics[trim={0 25 10 0},clip,width=0.33\linewidth]{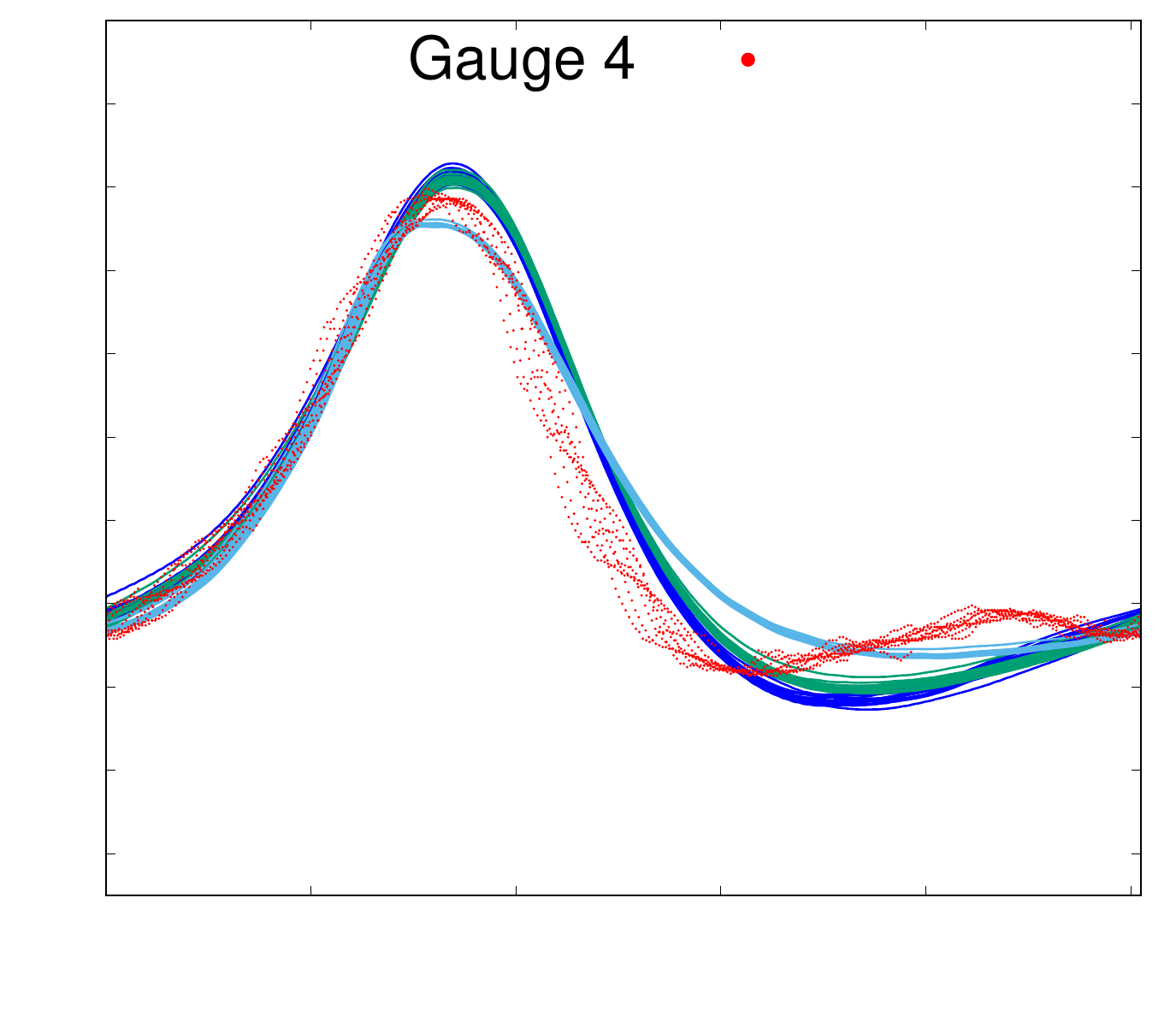} \\
    \includegraphics[trim={0 0 10 0},clip,width=0.33\linewidth]{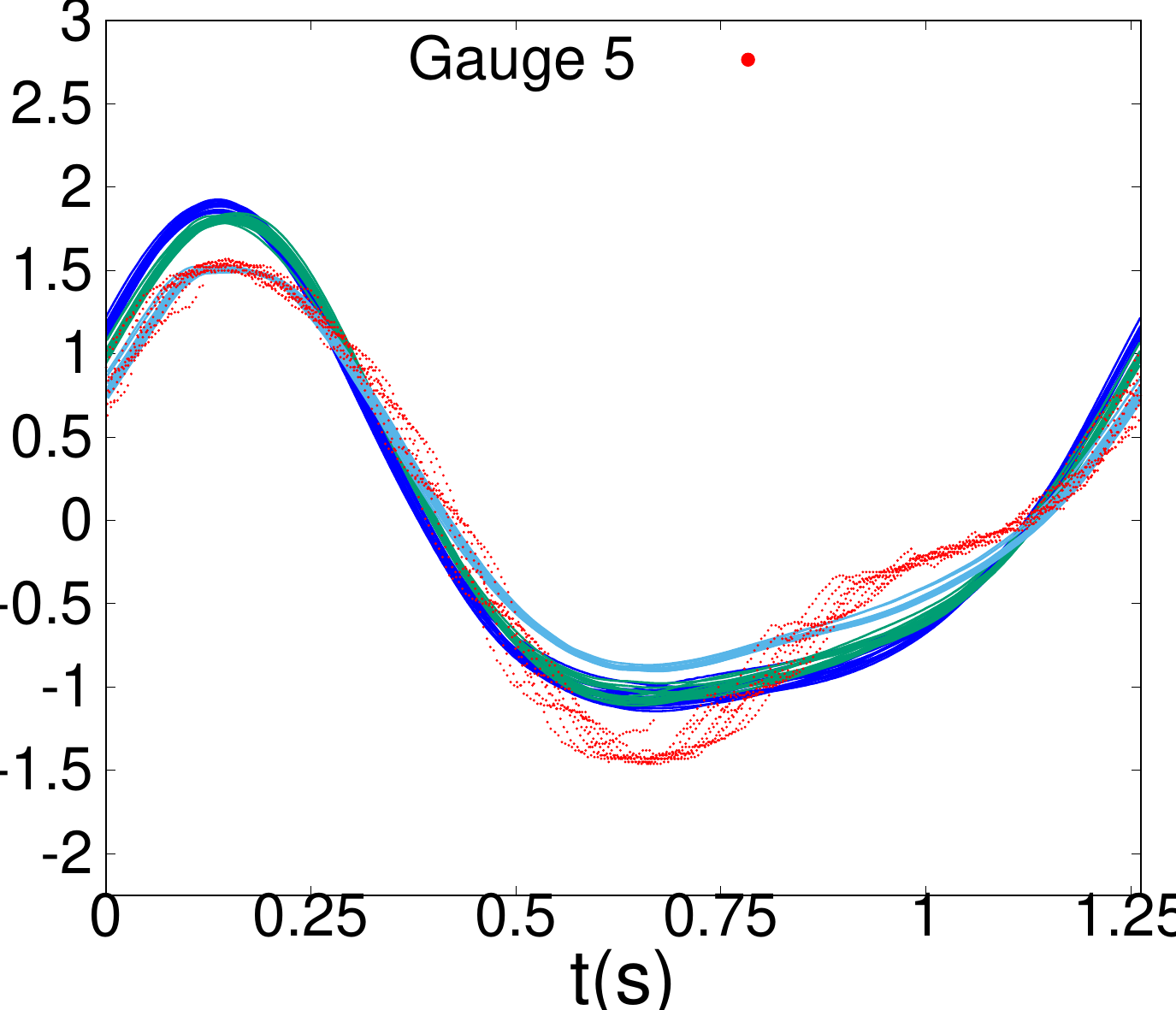} &
    \includegraphics[trim={0 0 10 0},clip,width=0.33\linewidth]{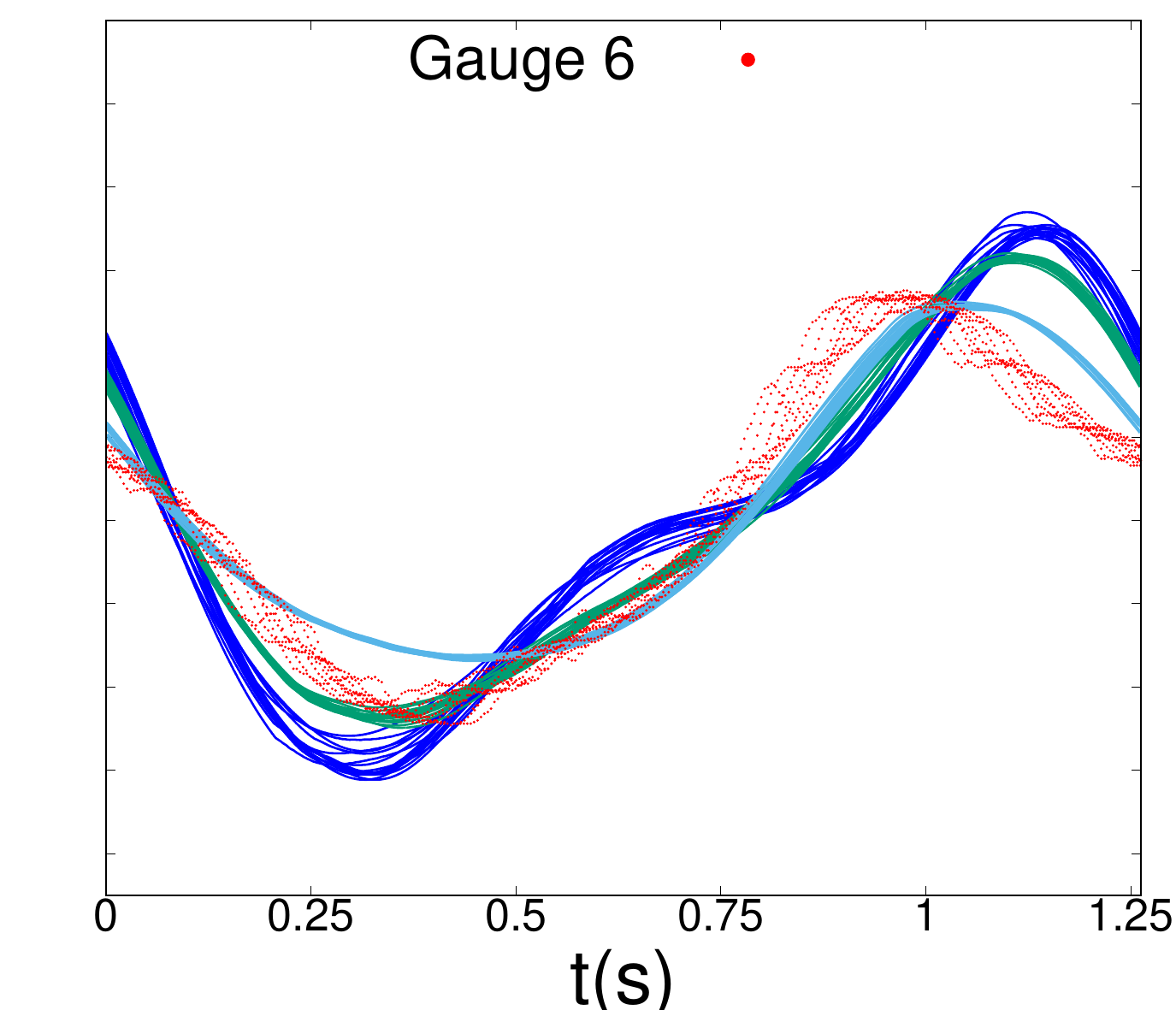} &
    \includegraphics[trim={0 0 10 0},clip,width=0.33\linewidth]{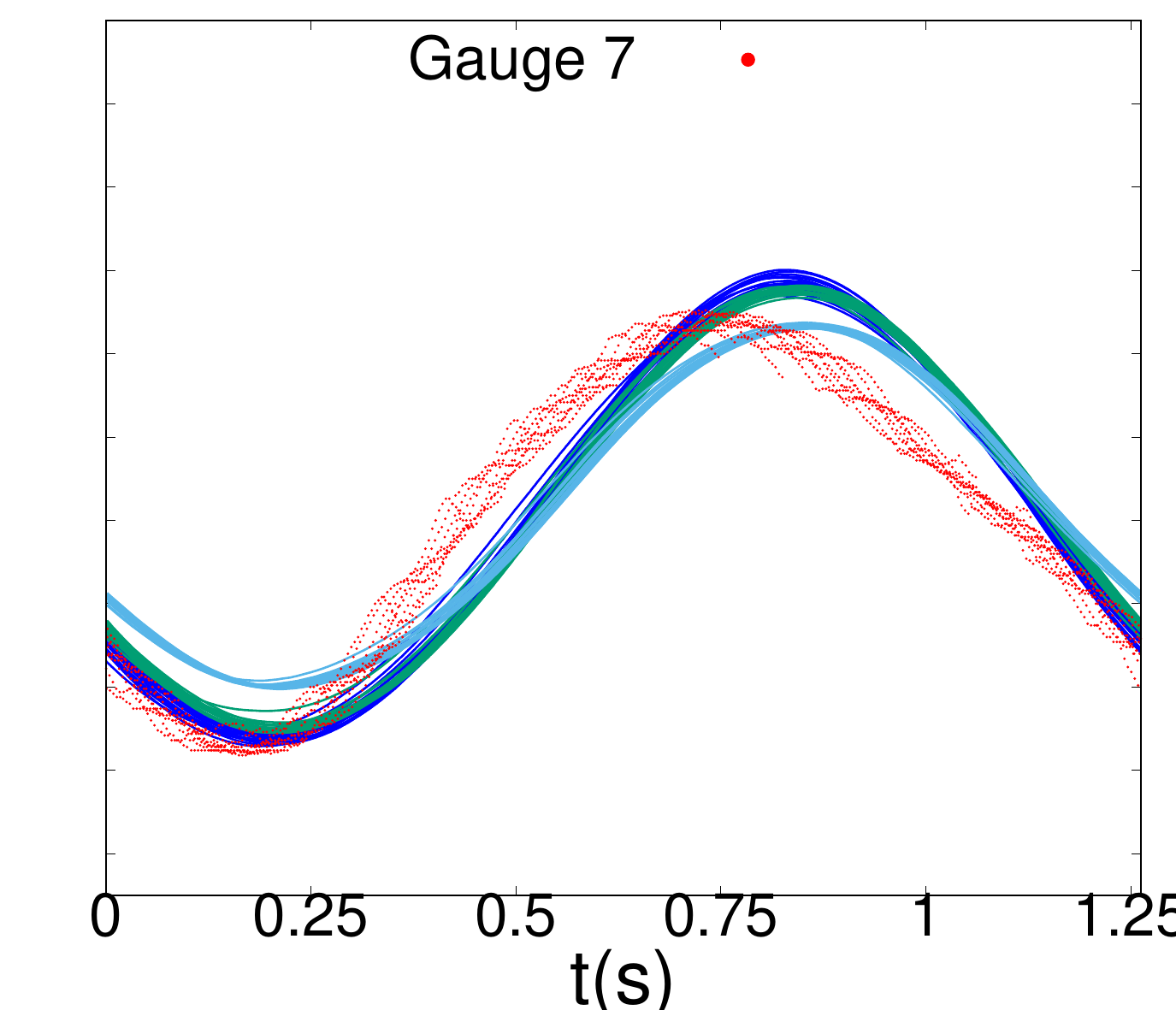}
    \end{tabular}
    \caption{SH Case. Water elevation at seven gauges. Numerical results using three meshes, $h = \{\SI{0.05}{m},\SI{0.025}{m},\SI{0.0125}{m}\}$ (solid lines). experimental data (red points).}\label{fig:sh_bar}%
\end{figure}

Using the adjusted values above and the period-folding technique, we
compare in Figures \ref{fig:sl_bar} and \ref{fig:sh_bar} the
experimental data and the results of the computations at the wave
gauges 2--7 (Figure~\ref{fig:sl_bar} for the SL experiment and and
Figure~\ref{fig:sh_bar} for the SH experiment). For the experimental
data, we choose $t_0$ to be the time corresponding to the second
maximum wave height of the signal at WG2.  For the numerical
simulations, we choose $t_0$ to be the time corresponding to the
maximum wave height around $t\approx\SI{40}{s}$ at WG2.
We observe that the numerical results
converge as the mesh is refined.
%We also see that
%in the two cases investigated the numerical simulation agree
%reasonably well with the experiments at gauges 2--5.
The SH
experiment is relatively well reproduced at all the gauges. There are slight deviations
at the last two gauges behind the bar for the SL experiment.
It is possible that some wave breaking occurs between gauges 5 and 6 in this case.
%These deviations are
%likely due to the Serre model not being able to capture the highly
%dispersive phenomena behind the bar.
We note here that the results
shown above are very similar to those seen in \cite{dumbser_2020} in
Figure 9 and Figure 10 therein.

\subsection{2D Solitary wave run-up over a conical island}
\label{Sec:conical_island}
We consider the 1995 laboratory experiments conducted by
\cite{Briggs1995} at the US Army Waterways Experiment Station in
Vicksburg, Mississippi (now the US Army Engineer Rsearch and Development Center). The laboratory experiments were motivated by
several tsunami events in the 1990s where large unexpected run-up
heights were observed on the back (or lee) side of small
islands. Several authors have used this experiment to study the run-up
phenomena using the classical Shallow Water model and other dispersive
models (see: \cite{HOU2013}, \cite{LANNES2015},
\cite{KAZOLEA_2012}).

Let $r(\bx)$ by the radius from the center of the island located at
$(\SI{12.96}{m},\SI{13.80}{m})$. Then the
conical island bathymetry is defined by
\begin{equation}
  z(\bx) =
  \begin{cases}
    \min\left(\waterh_{\text{top}}, \waterh_{\text{cone}}-r(\bx)/s_{\text{cone}}\right), & r(\bx) < r_{\text{cone}}\\
    0, & \text{otherwise}
\end{cases}
\label{eq:conical_island},
\end{equation}
where
$\waterh_{\text{top}}=0.625\si{m},\waterh_{\text{cone}}=0.9\si{m}$ and
$r_{\text{cone}}=3.6\si{m}$. We
reproduce two experiments, which we call Case B and Case C, with
$\alpha/\waterh_0 = 0.091$ and $\alpha/\waterh_0 = 0.181$ where
$\waterh_0 = \SI{0.32}{m}$ and $\alpha$  is the amplitude of the solitary wave.

The computations are done in
the domain $(0,\SI{25}{m})\times(0,\SI{30}{m})$ until the final time
$T=\SI{12}{s}$ with wall boundary conditions and CFL number
0.25. We initiate the solitary wave at
$x_0 = 9.36 - \frac{L}{2}$ using \eqref{eq:initial_solitary} with
$L = \frac{2\waterh_0}{k}\arccosh\sqrt{20}$ and $k = \sqrt{\frac{3\alpha}{4\waterh_0}}$.
%The recorded experimental wave heights were not consistent
%with the targeted amplitudes, so we initiate with
%$\alpha = 0.091\waterh_0,\, \alpha = 0.181\waterh_0$ for Case B and
%Case C, respectively.
Here $x_0$ is the
location of the experimental wave gauge 3 (WG3) which was used to measure
the free surface elevation away from the island. In Figure \ref{fig:island_figure},
we show the surface plots of the free surface elevation $\waterh + z$ on a
mesh composed of 52,129 $\bpolP_1$ nodes at $t=\{0,5.8,\SI{8}{s}\}$.
% Island 3d figure
 \begin{figure}[h]
 \centering
   \includegraphics[width= 0.32\linewidth]{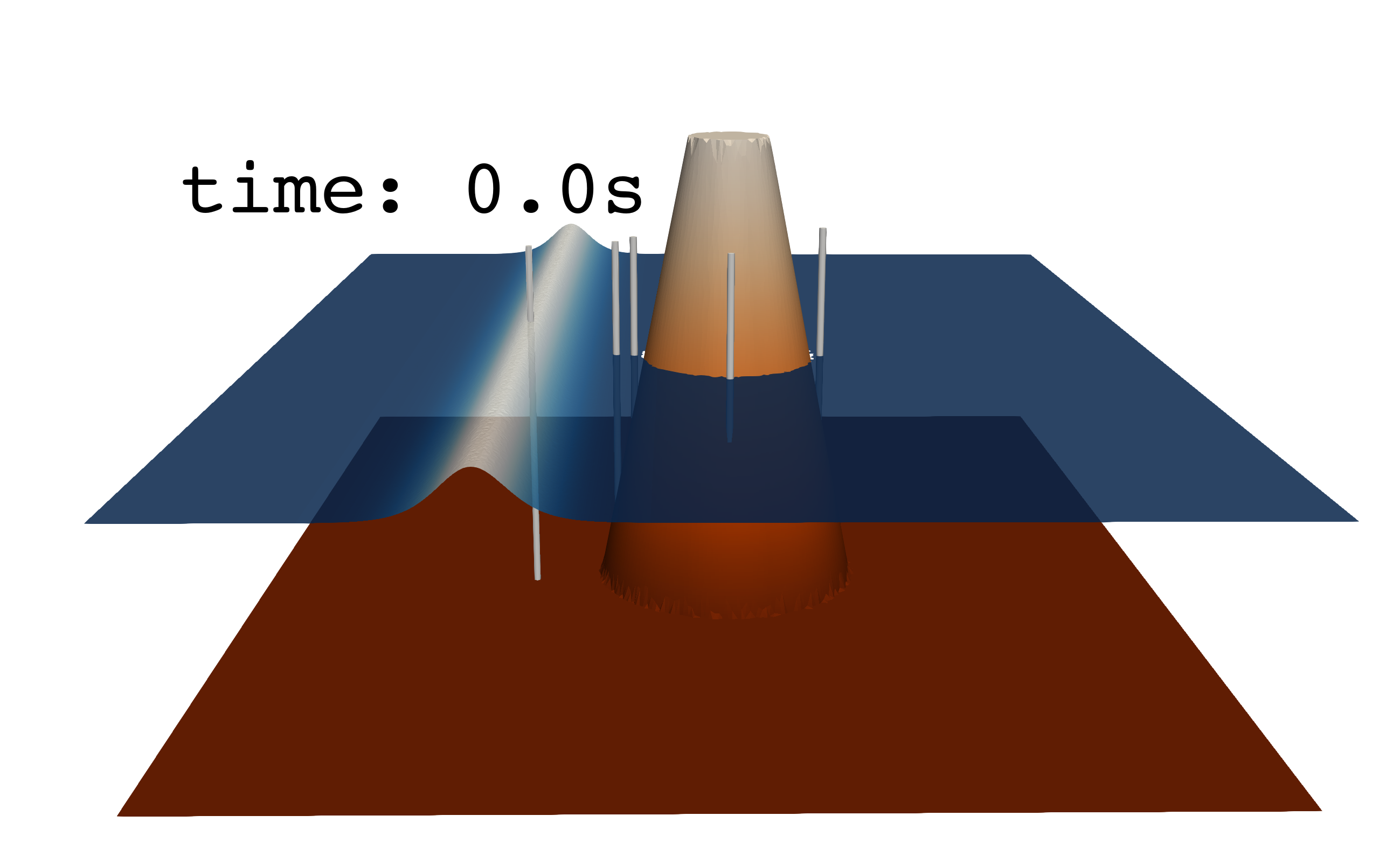}
   \includegraphics[width= 0.32\linewidth]{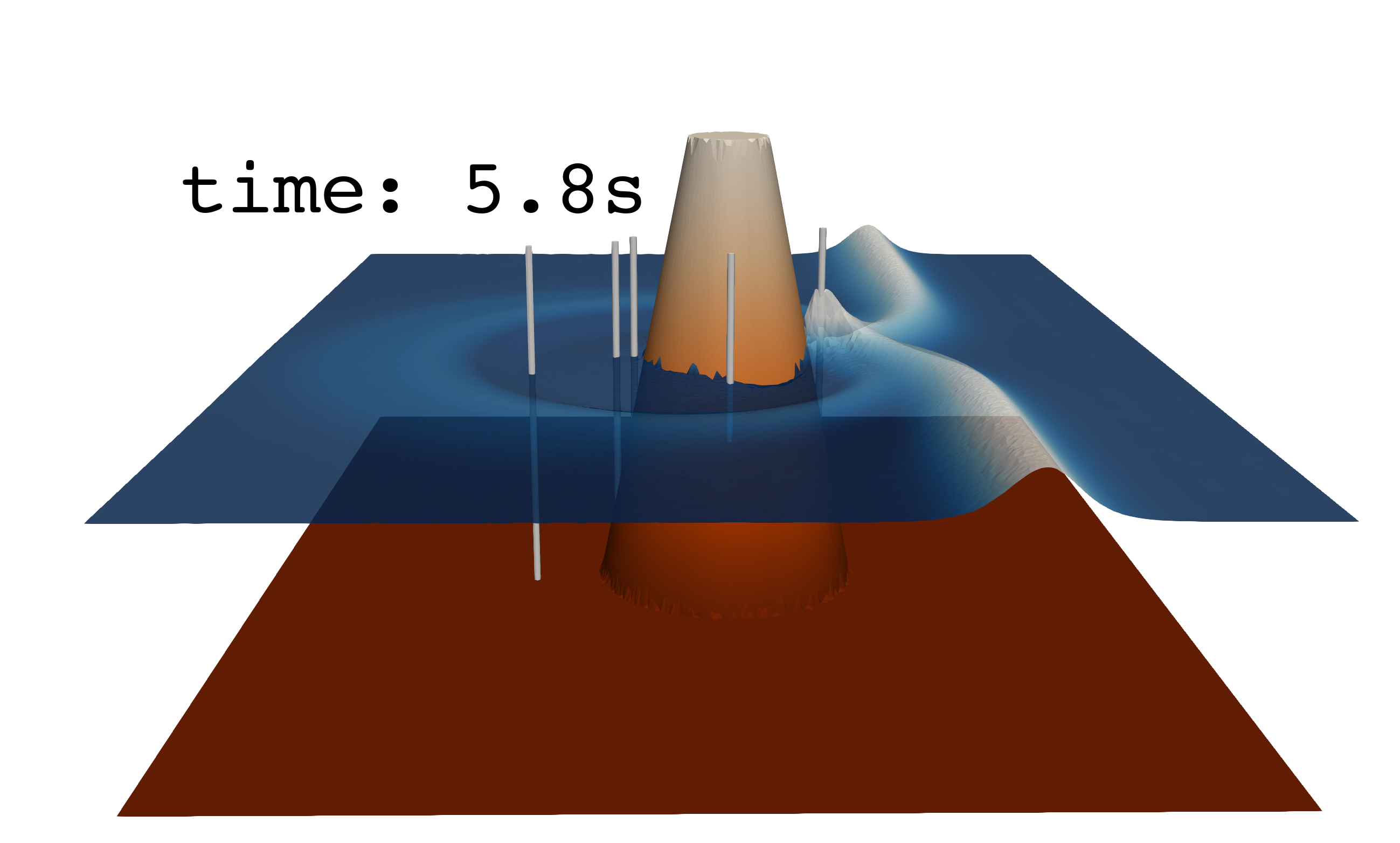}
   \includegraphics[width= 0.32\linewidth]{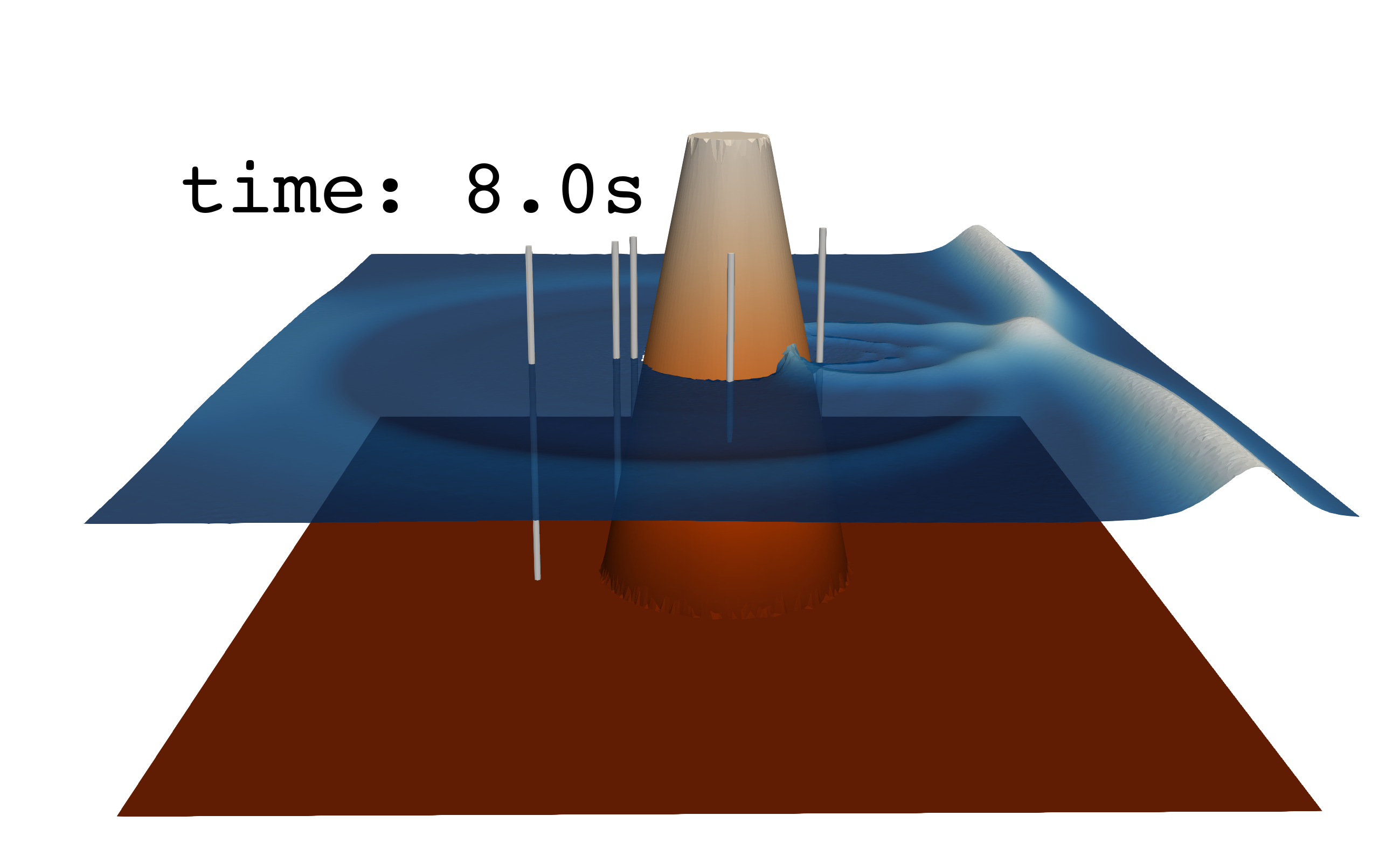}
   \caption{Experiment 4 -- Surface plot of the water elevation
     $\waterh + z$ at several times for Case C. The thin grey
     cylinders represent the wave gauges WG3, WG6, WG9, WG16, WG22 (left to right).}
 \label{fig:island_figure}
 \end{figure}

 In the experiment, several wave gauges were placed around the island
 to measure the free surface elevation and wave run-up. We compare the
 numerical results with the measurents at four of the experimental
 wave gauges: WG6$(\SI{9.36}{m},\SI{13.80}{m})$,
 WG9$(\SI{10.36}{m},\SI{13.80}{m})$,
 WG16$(\SI{12.96}{m},\SI{11.22}{m})$,
 WG22$(\SI{15.56}{m},\SI{13.80}{m})$. In Figure
 \ref{fig:island_gauges}, we show the comparison with the experimental
 data and numerical simulations for both Case B (on the left) and Case
 C (on the right). For both cases, the numerical results show good
 agreement with the experimental data. We capture well the magnitudes
 of the run-up and draw-down at the front side of the island at WG9
 with a slight overshoot in Case C. For both cases, we see very good
 comparison with WG16 which corresponds to the run-up and draw-down on
 the side of the island. We note that the experimental data shows
 subsequent free surface oscillations after impacting the island which
 is most notable in WG9, but our numerical simulation do not capture
 this effect. This phenomena is consistent with the literature and has
 been observed by others (see: \cite{LANNES2015, KAZOLEA_2012,
   Yamazaki_2009}), and is likely due to inconsistency in the original
 experiments.

% Island figure of free surface comparisons %
\begin{figure}[h]
\centering
    \includegraphics[trim={0 100 0 0},clip,width=0.45\linewidth]{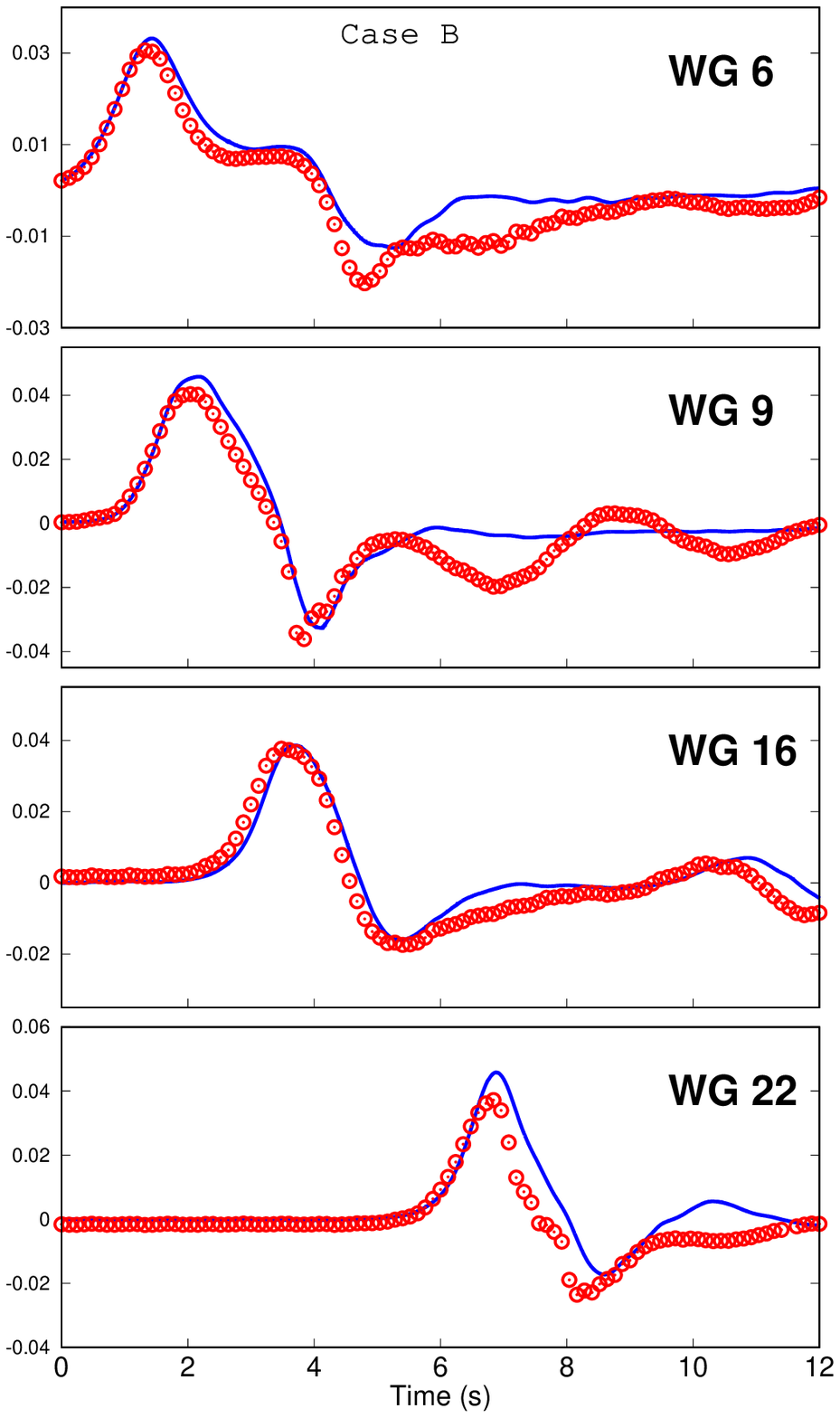}
    \includegraphics[trim={0 100 0 0},clip,width=0.45\linewidth]{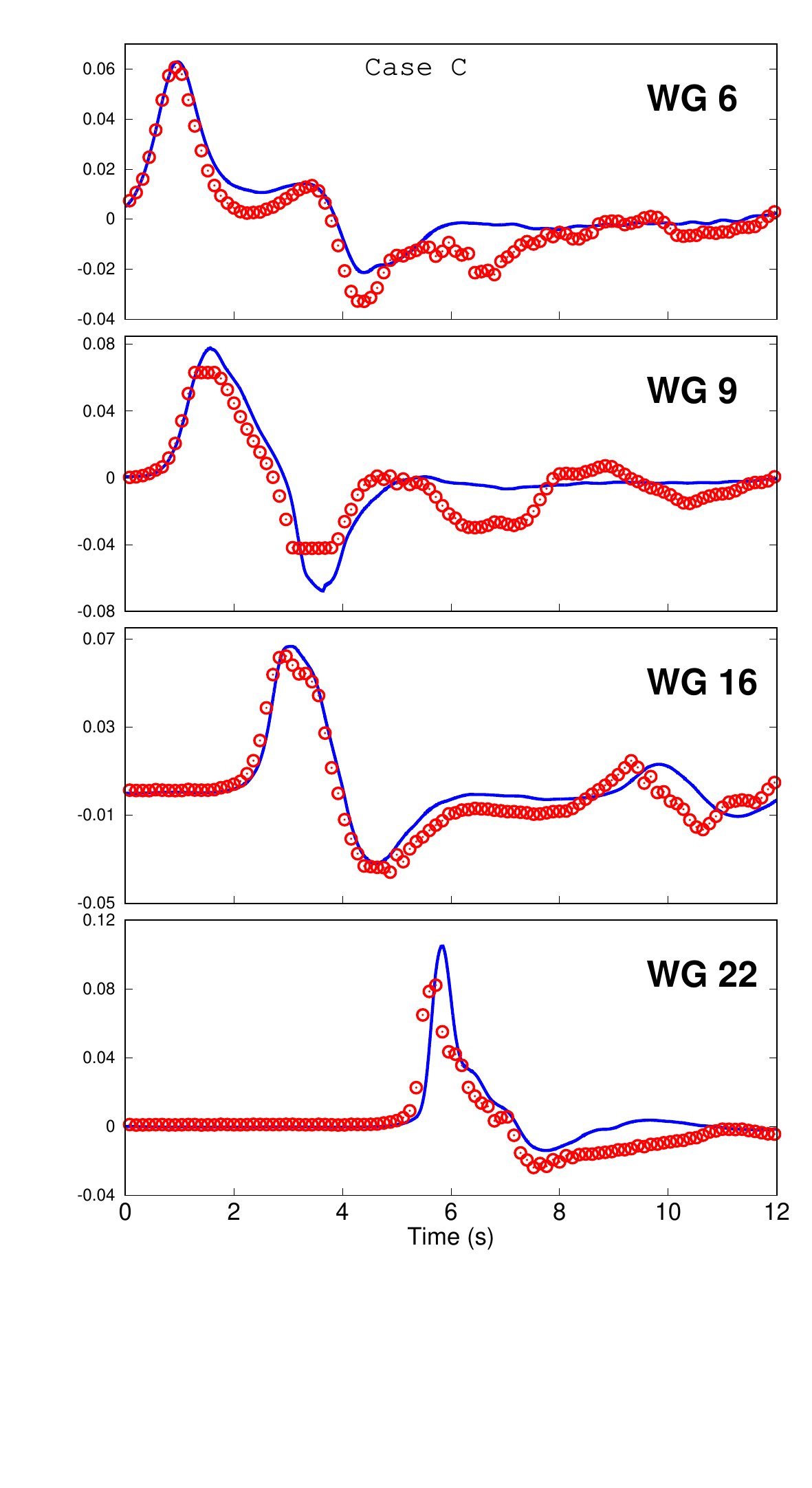}
    \caption{Experiment 4 -- Temporal series over the period
      $t\in[0,\SI{12}{s}]$ of the free surface elvation $\waterh + z$ in
      meters at the four WGs (blue solid) compared to the experimental
      data (red circles) for Case B (on the
      left) and Case C (on the
      right).}\label{fig:island_gauges}%
\end{figure}
%%%
\section{Conclusion}
\label{sec:conclusion}
In this paper we introduced a relaxation technique for solving the dispersive Serre
equations with full effects induced by the topography. We also derived a family
of analytical solutions to the dispersive Serre equations that can be used
for validating numerical methods. The relaxation approach yields a new hyperbolic system that is
compatible with dry states and extends the work presented in
%\cite{Favrie_Gavrilyuk_2017} and
\citep{Gu_Po_To_Ke_FAKE_SGN_2019}.
This hyperbolic system is then shown numerically to converge to the
original dispersive Serre model at the expected first-order
convergence rate when the relaxation parameter is chosen to be
proportional to the local mesh-size.
We then compared our numerical computations with several laboratory experiments for model validation and demonstrated close agreement with said data. We also
showed that neglecting the full terms induced by the topography as was
done in \citep{Gu_Po_To_Ke_FAKE_SGN_2019} yields poor agreement with
the experimental data and thus shows the importance of these terms.

%%%%%%%%%%%%%%%
\bibliographystyle{abbrvnat}
\bibliography{ref}

\begin{thebibliography}{29}
\providecommand{\natexlab}[1]{#1}
\providecommand{\url}[1]{\texttt{#1}}
\expandafter\ifx\csname urlstyle\endcsname\relax
  \providecommand{\doi}[1]{doi: #1}\else
  \providecommand{\doi}{doi: \begingroup \urlstyle{rm}\Url}\fi

\bibitem[Alvarez-Samaniego and Lannes(2008)]{Alvarez-Samaniego_Lannes_2008B}
B.~Alvarez-Samaniego and D.~Lannes.
\newblock Large time existence for 3{D} water-waves and asymptotics.
\newblock \emph{Invent. Math.}, 171\penalty0 (3):\penalty0 485--541, 2008.

\bibitem[Barth\'elemy(2004)]{Barthelemy_2004}
E.~Barth\'elemy.
\newblock Nonlinear shallow water theories for coastal waves.
\newblock \emph{Surveys in Geophysics}, 25:\penalty0 315--337, 2004.

\bibitem[Bassi et~al.(2020)Bassi, Bonaventura, Busto, and
  Dumbser]{dumbser_2020}
C.~Bassi, L.~Bonaventura, S.~Busto, and M.~Dumbser.
\newblock A hyperbolic reformulation of the {S}erre--{G}reen--{N}aghdi model
  for general bottom topographies.
\newblock \emph{Computers \& Fluids}, 212:\penalty0 104716, 2020.

\bibitem[Beji and Battjes(1994)]{BEJI_1994}
S.~Beji and J.~Battjes.
\newblock Numerical simulation of nonlinear wave propagation over a bar.
\newblock \emph{Coastal Engineering}, 23\penalty0 (1):\penalty0 1 -- 16, 1994.

\bibitem[Bonneton et~al.(2011)Bonneton, Chazel, Lannes, Marche, and
  Tissier]{Bonneton_etal_2011}
P.~Bonneton, F.~Chazel, D.~Lannes, F.~Marche, and M.~Tissier.
\newblock A splitting approach for the fully nonlinear and weakly dispersive
  {G}reen--{N}aghdi model.
\newblock \emph{J. Comput. Phys.}, 230\penalty0 (4):\penalty0 1479--1498, 2011.

\bibitem[Briggs et~al.(1995)Briggs, Synolakis, Harkins, and Green]{Briggs1995}
M.~J. Briggs, C.~E. Synolakis, G.~S. Harkins, and D.~R. Green.
\newblock Laboratory experiments of tsunami runup on a circular island.
\newblock \emph{pure and applied geophysics}, 144\penalty0 (3):\penalty0
  569--593, 1995.

\bibitem[Bristeau et~al.(2015)Bristeau, Mangeney, Sainte-Marie, and
  Seguin]{Bristeau_Mangenay_SaniteMarie_Seguin_2015}
M.-O. Bristeau, A.~Mangeney, J.~Sainte-Marie, and N.~Seguin.
\newblock An energy-consistent depth-averaged {E}uler system: derivation and
  properties.
\newblock \emph{Discrete Contin. Dyn. Syst. Ser. B}, 20\penalty0 (4):\penalty0
  961--988, 2015.

\bibitem[Castro and Lannes(2014)]{castro_lannes_2014}
A.~Castro and D.~Lannes.
\newblock Fully nonlinear long-wave models in the presence of vorticity.
\newblock \emph{Journal of Fluid Mechanics}, 759:\penalty0 642–675, 2014.

\bibitem[Duran and Marche(2017)]{Duran_Marche_2017}
A.~Duran and F.~Marche.
\newblock A discontinuous {G}alerkin method for a new class of
  {G}reen--{N}aghdi equations on simplicial unstructured meshes.
\newblock \emph{Appl. Math. Model.}, 45:\penalty0 840--864, 2017.

\bibitem[Escalante et~al.(2019)Escalante, Dumbser, and
  Castro]{Escalante_Dumbser_Castro_2019}
C.~Escalante, M.~Dumbser, and M.~J. Castro.
\newblock An efficient hyperbolic relaxation system for dispersive
  non-hydrostatic water waves and its solution with high order discontinuous
  {G}alerkin schemes.
\newblock \emph{J. Comput. Phys.}, 394:\penalty0 385--416, 2019.

\bibitem[Favrie and Gavrilyuk(2017)]{Favrie_Gavrilyuk_2017}
N.~Favrie and S.~Gavrilyuk.
\newblock A rapid numerical method for solving {S}erre--{G}reen--{N}aghdi
  equations describing long free surface gravity waves.
\newblock \emph{Nonlinearity}, 30\penalty0 (7):\penalty0 2718--2736, 2017.

\bibitem[Fernandez-Nieto et~al.(2019)Fernandez-Nieto, Parisot, Penel, and
  Sainte-Marie]{nieto_parisot_2018}
E.~D. Fernandez-Nieto, M.~Parisot, Y.~Penel, and J.~Sainte-Marie.
\newblock A hierarchy of dispersive layer-averaged approximations of euler
  equations for free surface flows.
\newblock \emph{Communications in Mathematical Sciences}, 16\penalty0
  (5):\penalty0 1169--1202, 2019.

\bibitem[Gavrilyuk and Shugrin(1996)]{Gavrilyuk_1996}
S.~L. Gavrilyuk and S.~M. Shugrin.
\newblock Media with equations of state that depend on derivatives.
\newblock \emph{Journal of Applied Mechanics and Technical Physics},
  37\penalty0 (2):\penalty0 177--189, 1996.

\bibitem[Green and Naghdi(1976)]{green_naghdi_1976}
A.~E. Green and P.~M. Naghdi.
\newblock A derivation of equations for wave propagation in water of variable
  depth.
\newblock \emph{Journal of Fluid Mechanics}, 78\penalty0 (2):\penalty0
  237–246, 1976.
\newblock \doi{10.1017/S0022112076002425}.

\bibitem[Green et~al.(1974)Green, Laws, and Naghdi]{Green_Naghdi_1974}
A.~E. Green, N.~Laws, and P.~M. Naghdi.
\newblock On the theory of water waves.
\newblock \emph{Proc. Roy. Soc. (London) Ser. A}, 338:\penalty0 43--55, 1974.

\bibitem[Guermond et~al.(2018)Guermond, Quezada~de Luna, Popov, Kees, and
  Farthing]{Guermond_Quezada_Popov_Kees_Farthing_2018}
J.-L. Guermond, M.~Quezada~de Luna, B.~Popov, C.~Kees, and M.~Farthing.
\newblock Well-balanced second-order finite element approximation of the
  shallow water equations with friction.
\newblock \emph{SIAM Journal on Scientific Computing}, 40\penalty0
  (6):\penalty0 A3873--A3901, 2018.

\bibitem[Guermond et~al.(2019)Guermond, Popov, Tovar, and
  Kees]{Gu_Po_To_Ke_FAKE_SGN_2019}
J.-L. Guermond, B.~Popov, E.~Tovar, and C.~Kees.
\newblock Robust explicit relaxation technique for solving the
  {G}reen--{N}aghdi equations.
\newblock \emph{J. Comput. Phys.}, 399:\penalty0 108917, 17, 2019.

\bibitem[Guibourg(1994)]{guibo1994}
S.~Guibourg.
\newblock \emph{Modélisations numérique et expérimentale des houles
  bidimensionnelles en zone cotière}.
\newblock PhD thesis, 1994.
\newblock URL \url{http://www.theses.fr/1994GRE10160}.

\bibitem[Hou et~al.(2013)Hou, Liang, Simons, and Hinkelmann]{HOU2013}
J.~Hou, Q.~Liang, F.~Simons, and R.~Hinkelmann.
\newblock A stable 2d unstructured shallow flow model for simulations of
  wetting and drying over rough terrains.
\newblock \emph{Computers \& Fluids}, 82:\penalty0 132 -- 147, 2013.

\bibitem[Kazolea et~al.(2012)Kazolea, Delis, Nikolos, and
  Synolakis]{KAZOLEA_2012}
M.~Kazolea, A.~Delis, I.~Nikolos, and C.~Synolakis.
\newblock An unstructured finite volume numerical scheme for extended 2d
  boussinesq-type equations.
\newblock \emph{Coastal Engineering}, 69:\penalty0 42 -- 66, 2012.

\bibitem[Lannes(2020)]{Lannes_2020}
D.~Lannes.
\newblock Modeling shallow water waves.
\newblock \emph{Nonlinearity}, 33\penalty0 (5):\penalty0 R1--R57, 2020.

\bibitem[Lannes and Marche(2015)]{LANNES2015}
D.~Lannes and F.~Marche.
\newblock A new class of fully nonlinear and weakly dispersive
  {G}reen--{N}aghdi models for efficient 2d simulations.
\newblock \emph{Journal of Computational Physics}, 282:\penalty0 238 -- 268,
  2015.

\bibitem[Samii and Dawson(2018)]{samii_dawson_2020}
A.~Samii and C.~Dawson.
\newblock An explicit hybridized discontinuous galerkin method for
  {S}erre--{G}reen--{N}aghdi wave model.
\newblock \emph{Computer Methods in Applied Mechanics and Engineering},
  330:\penalty0 447 -- 470, 2018.

\bibitem[Seabra-Santos et~al.(1987)Seabra-Santos, Renouard, and
  Temperville]{seabra-santos_renouard_temperville_1987}
F.~J. Seabra-Santos, D.~P. Renouard, and A.~M. Temperville.
\newblock Numerical and experimental study of the transformation of a solitary
  wave over a shelf or isolated obstacle.
\newblock \emph{Journal of Fluid Mechanics}, 176:\penalty0 117–134, 1987.

\bibitem[Serre(1953)]{Serre_1953}
F.~Serre.
\newblock Contribution \`a l'\'etude des \'ecoulements permanents et variables
  dans les canaux.
\newblock \emph{La Houille Blanche}, \penalty0 (6):\penalty0 830--872, 1953.
\newblock URL \url{https://doi.org/10.1051/lhb/1953058}.

\bibitem[Su and Gardner(1969)]{Su_Gardner_1969}
C.~H. Su and C.~S. Gardner.
\newblock Korteweg-de {V}ries equation and generalizations. {III}. {D}erivation
  of the {K}orteweg-de {V}ries equation and {B}urgers equation.
\newblock \emph{J. Mathematical Phys.}, 10:\penalty0 536--539, 1969.

\bibitem[Tkachenko(2020)]{TKACHENKO_thesis}
S.~Tkachenko.
\newblock \emph{Analytical and numerical study of a dispersive shallow water
  model}.
\newblock PhD thesis, Aix-Marseille University, 2020.

\bibitem[Tovar(2021)]{Tovar_PhD}
E.~Tovar.
\newblock \emph{Well-balanced and invariant domain preserving schemes for
  dispersive shallow water flows}.
\newblock PhD thesis, Texas A\&M, 2021.
\newblock In preparation.

\bibitem[Yamazaki et~al.(2009)Yamazaki, Kowalik, and Cheung]{Yamazaki_2009}
Y.~Yamazaki, Z.~Kowalik, and K.~F. Cheung.
\newblock Depth-integrated, non-hydrostatic model for wave breaking and run-up.
\newblock \emph{International Journal for Numerical Methods in Fluids},
  61\penalty0 (5):\penalty0 473--497, 2009.

\end{thebibliography}

\end{document}